\documentclass[11pt]{article}
\usepackage{microtype}
\usepackage{fullpage}
\usepackage{graphicx}
\usepackage{hyperref}

\usepackage{amsmath,amsfonts,bm}

\def\1{\bm{1}}

\def\vb{{\bm{b}}}
\def\vc{{\bm{c}}}

\def\ve{{\bm{e}}}

\def\vs{{\bm{s}}}

\def\vu{{\bm{u}}}
\def\vv{{\bm{v}}}

\def\vx{{\bm{x}}}
\def\vy{{\bm{y}}}

\def\mA{{\bm{A}}}
\def\mB{{\bm{B}}}
\def\mC{{\bm{C}}}
\def\mD{{\bm{D}}}
\def\mE{{\bm{E}}}
\def\mF{{\bm{F}}}

\def\mH{{\bm{H}}}
\def\mI{{\bm{I}}}
\def\mJ{{\bm{J}}}

\def\mL{{\bm{L}}}

\def\mU{{\bm{U}}}

\def\mW{{\bm{W}}}
\def\mX{{\bm{X}}}

\DeclareMathAlphabet{\mathsfit}{\encodingdefault}{\sfdefault}{m}{sl}
\SetMathAlphabet{\mathsfit}{bold}{\encodingdefault}{\sfdefault}{bx}{n}

\def\gD{{\mathcal{D}}}

\def\gI{{\mathcal{I}}}

\def\gS{{\mathcal{S}}}

\def\sN{{\mathbb{N}}}

\def\sP{{\mathbb{P}}}

\def\sR{{\mathbb{R}}}

\newcommand{\E}{\mathbb{E}}

\newcommand{\R}{\mathbb{R}}

\def\tr{\mathrm{tr}}
\usepackage[utf8]{inputenc} % allow utf-8 input
\usepackage[T1]{fontenc}    % use 8-bit T1 fonts
\usepackage{url}            % simple URL typesetting
\usepackage{booktabs}       % professional-quality tables
\usepackage{amsfonts}       % blackboard math symbols
\usepackage{nicefrac}       % compact symbols for 1/2, etc.
\usepackage{microtype}      % microtypography
\usepackage{mathrsfs}
\usepackage{multirow}
\usepackage{algorithm}
\usepackage{algorithmic}
\usepackage{adjustbox}
\usepackage{bbm}
\usepackage{bm}
\usepackage{color}
\usepackage{mathrsfs}
\usepackage{amsmath} %数学公式
\usepackage{amsthm}
\usepackage[numbers]{natbib}

\usepackage{float}
\usepackage{diagbox}
\usepackage{subcaption}
\usepackage{amssymb}
\usepackage[dvipsnames]{xcolor}
\usepackage{enumitem}
\newtheorem{thm}{Theorem}[section]
\newtheorem{defn}{Definition}
\newtheorem{lemma}[thm]{Lemma}
\newtheorem{assume}{Assumption}
\newtheorem{remark}[thm]{Remark}
\newtheorem{corollary}[thm]{Corollary}
\newtheorem{prop}{Proposition}

\usepackage[toc,page]{appendix}
\usepackage{tikz}
\usepackage{blkarray}

\newcommand{\norm}[1]{\left\|#1\right\|}

\def\tr{\mathrm{tr}}

\begin{document}
\title{Greedy and Random Broyden's Methods with Explicit Superlinear Convergence Rates in Nonlinear Equations}
\author{Haishan Ye\thanks{Equal Contribution.} 
	\thanks{School of Management; Xi'an Jiaotong University;
		\texttt{hsye\_cs@outlook.com};}
	\and
	Dachao Lin\footnotemark[1] 
	\thanks{Academy for Advanced Interdisciplinary Studies;
		Peking University;
		\texttt{lindachao@pku.edu.cn};}	
	\and
	Zhihua Zhang 
	\thanks{School of Mathematical Sciences;
		Peking University;
		\texttt{zhzhang@math.pku.edu.cn}.}
}
\maketitle
\begin{abstract}
    In this paper, we propose the greedy and random Broyden's method for solving nonlinear equations.
    Specifically, the greedy method greedily selects the direction to maximize a certain measure of progress for approximating the current Jacobian matrix, while the random method randomly chooses a direction.
    We establish explicit (local) superlinear convergence rates of both methods if the initial point and approximate Jacobian are close enough to a solution and corresponding Jacobian.
    Our two novel variants of Broyden's method enjoy two important advantages that the approximate Jacobians of our algorithms will converge to the exact ones and the convergence rates of our algorithms are asymptotically faster than the original Broyden's method.
    Our work is the first time to achieve such two advantages theoretically.  
    Our experiments also empirically validate the advantages of our algorithms.
\end{abstract}

\section{Introduction}

In this paper, we mainly consider the following general nonlinear equation system
\begin{equation}
    \label{obj}
    \mF(\vx) = \bm{0},
\end{equation}
where $\mF: \sR^n \to \sR^n$ is a differentiable vector-valued function. 
This equation appears in many areas, such as control systems \cite{more1989collection}, economic systems \cite{albu2006non}, stationary points of optimization problems, etc.
Many numerical methods have been proposed to solve the nonlinear equation, since obtaining an analytical solution is hard.
Newton's method is a classical method for this problem based on the following iteration:
\[ \vx_{k+1} = \vx_{k} - \left[\mJ(\vx_{k})\right]^{-1}\mF(\vx_{k}), \]
where $\mJ(\vx)\in \sR^{n\times n}$ is the exact Jabobian of $\mF(\vx)$. 
The convergence is quadratic under some mild conditions if $\vx_0$ is sufficiently close to a solution \cite{nocedal2006numerical}.
However, the original Newton's method suffers from expensive computation cost, particularly due to the inverse calculation of the Jacobian which takes $O(n^3)$ flops.
To remedy the heavy computation cost of Newton's method, the quasi-Newton methods \cite{broyden1970convergence2, broyden1970convergence,fletcher1970new, goldfarb1970family, shanno1970conditioning, davidon1991variable, fletcher1963rapidly, broyden1965class,broyden1967quasi,davidon1991variable} are proposed which adopt $\mB_k$ as an approximation of the Jacobian $\mJ(\vx_k)$:
\begin{equation*}
\label{eq:x_up}
\vx_{k+1} = \vx_{k} - \mB_k^{-1}\mF(\vx_{k}),
\end{equation*} 
and update it to mimic the behavior of the true Jacobian at each iteration.
There have been several quasi-Newton methods proposed for solving nonlinear equations such as the stationary
Newton's method \cite{martinez2000practical}, Newton’s method with ``$p$ refinements'' \cite{martinez2000practical}, the discrete
Newton's method \cite{martinez2000practical}, the Broyden's method \cite{broyden1965class}, ABS \cite{abaffy1989abs}, etc.
Among these methods, Broyden’s method, which we show in Algorithm \ref{algo:broyden}, is one of the most prominent quasi-Newton methods for solving nonlinear equations, and there is an ever-growing body of literature available.
The surveys \cite{al2014broyden, griewank2012broyden, martinez2000practical} cover many aspects and provide further references and open problems.
Several works also concerned the extension of Broyden's method to constrained nonlinear systems of equations \cite{marini2018quasi}, set-valued mappings \cite{artacho2014local} and implicit deep learning \cite{bai2019deep, bai2020multiscale}.
% In this paper, we focus on Broyden's method for solving nonlinear equations.

\begin{algorithm}[t]
	\caption{Broyden's Method}
	\begin{algorithmic}[1]
		\STATE Initialization: set $\mB_0, \vx_0$.
		\FOR{$ k \geq 0 $}
		\STATE Update $\vx_{k+1} = \vx_{k} - \mB_k^{-1} \mF(\vx_k)$.
		\STATE Set $\vu_{k}=\vx_{k+1}-\vx_k$, $\vy_{k}= \mF(\vx_{k+1})- \mF(\vx_{k})$.
		\STATE Compute 
		$\mB_{k+1} =  \mB_k + \frac{\left(\vy_k-\mB_k\vu_k\right)\vu_k^\top}{\vu_k^\top\vu_k}$.
		\ENDFOR
	\end{algorithmic}
	\label{algo:broyden}
\end{algorithm}

Many works have shown that Broyden's method can achieve a local superlinear convergence rate of $\{\vx_k\} $ \cite{broyden1973local, kelley1991new, dennis1996numerical, griewank1987local}. 
However, the convergence of approximate Jacobian $ \{\mB_k\}  $ is still largely open \cite{mannel2021convergence} and several researchers have mentioned this issue in their works \cite{dennis1977quasi, martinez2000practical, griewank2012broyden}. 
The convergence of $ \{\mB_k\}  $ is for example of interest because it is closely
related to the rate of convergence of $\{\vx_k\} $ \cite{mannel2021convergence}.
The only results are established in \cite{more1976global, mannel2021convergence}, which required uniformly linearly independent $\vu_k$'s. 
Unfortunately, conditions that imply uniform linear independence of $\vu_k$'s are unknown.
Moreover, we also discover the gap between $\mB_k$ and $\mJ_k$ in our numerical experiments.
Hence, we try to propose some variants of Broyden's method to achieve the goal that the approximate Jacobian will converge to the exact one.
Recently, \citet{lin2021explicit} first gave the explicit superlinear convergence rate of the original Broyden's method.
It is a natural question whether one can design a faster quasi-Newton method than the original Broyden's method.

In this paper, we try to solve the above two open problems.
Revisiting the classical Broyden's method and inspired by the greedy and random quasi-Newton for minimizing convex functions \cite{rodomanov2021greedy,lin2021faster},  we propose two novel variants of Broyden's method whose approximate Jacobians will converge to the exact one and they can achieve asymptotic faster convergence rates than the original Broyden's method.
We summarize our contribution as follows. 
\begin{enumerate}
	\item We propose a greedy update of the quasi-Jacobian matrix motivated by maximizing a certain measure of progress. 
	Moreover, we also suggest a random update from common distributions which achieves similar progress in expectation.
	Accordingly, we present greedy and random Broyden's method utilizing such update rules.
	\item We show that if the initial point $\vx_0$ is sufficiently close to one solution $\vx^*$ and
	the initial Jacobian approximation error $\norm{\mB_0-\mJ(\vx^*)}$ is sufficiently small, where matrix $\mB_0\in\R^{n\times n}$ is the initial approximate Jacobian, then the sequence $\{\vx_k\} $ produced by the greedy or random Broyden's methods could converge to the solution $\vx^*$ at a superlinear rate of the form $\left(1-\nicefrac{1}{n}\right)^{k(k-1)/4}$, and $\mB_k$ converges to $\mJ(\vx_k)$ at a linear rate of the form $\left(1-\nicefrac{1}{n}\right)^{k/2}$.
	\item Our methods offer flexible choices of update directions as well as superlinear convergence guarantees. We also verify our theory through numerical experiments.
\end{enumerate}

\paragraph{Organization.}
The paper is organized as follows. After reviewing related work in Section~\ref{sec:related}, we establish some
preliminaries in Section~\ref{sec:preliminaries}, including showing notation and reviewing Broyden's update.
In Section~\ref{sec:gre-ran-method}, we introduce the greedy and random updates and propose greedy and random Broyden's method accordingly. We also present our main results of explicit local superlinear convergence rates of our two algorithms.
Section~\ref{sec:analysis} provides the detailed convergence analysis of greedy and random Broyden's methods.
In Section~\ref{sec:exp}, we perform numerical experiments to support our theoretical results.
Finally, we conclude our work in Section~\ref{sec:conclusion}.

\section{Related Work}\label{sec:related}

Quasi-Newton methods for solving nonlinear equations have been widely studied and several important algorithms have been proposed \cite{broyden1965class,thomas1975sequential,abaffy1989abs,spedicato1978some}.
Among all quasi-Newton methods for solving nonlinear equations, Broyden's method is the most famous and important \cite{broyden1965class,broyden1967quasi,kelley1991new,nocedal2006numerical}.
For each iteration, Broyden's method conducts a rank-one update for approximating the true Jacobian, and the updated matrix has to satisfy the secant equation.
\citet{broyden1973local} provided the first superlinear convergent rate for Broyden's method.
After that, many works gave improved analysis for Broyden's methods \cite{dennis1996numerical,gruver1981algorithmic,kelley1991new,hwang1992convergence,kelley1995iterative,lin2021explicit}. 

Recently, for the unconstrained optimization, i.e., $\min_{\vx} f(\vx)$, \citet{rodomanov2021greedy} gave the first explicit superlinear convergence rate for the greedy quasi-Newton methods which take updates by greedily selecting from basis vectors to maximize a certain measure of progress.
This work established an explicit non-asymptotic rate of the local superlinear convergence $(1-\nicefrac{1}{n\kappa})^{k^2/2}$, where $k$ is the iteration numbers and $\kappa$ is the condition number of the objective function trying to minimize.
\citet{lin2021faster} improved the rates of greedy quasi-Newton updates as well as random quasi-Newton updates and obtained a faster condition-number-free superlinear convergence rate $(1-\nicefrac{1}{n})^{k^2/2}$.
Our proposed methods are similar to theirs but mainly applied to general nonlinear equations under Broyden's methods. 
Since Broyden's method can also be applied to such cases by solving $\nabla f(\vx)=\bm{0}$, as well as other important problems, such as fixed-point problems: $\vx-\mF(\vx)=\bm{0}$.
Therefore, our variants of Broyden's method for the general nonlinear equation are still meaningful. 
 
%As for classical quasi-Newton methods, \citet{rodomanov2021rates} also analyzed the classical well-known DFP and BFGS methods, adopting standard Hessian update direction through the previous variation. 
%They demonstrated the rates of the form $(\frac{n\kappa^2}{k})^{k/2}$ and $(\frac{n\kappa}{k})^{k/2}$ for the standard DFP and BFGS methods, respectively. 
%Such rates have faster initial convergence rates, while slower final rates compared to \citet{rodomanov2021greedy}'s results.
%Recently, \citet{rodomanov2021new} improved the results of \citet{rodomanov2021rates} and reduce the dependence of the condition number from $\kappa$ to $\ln\kappa$, though having similar worse long-history behavior compared to \cite{rodomanov2021greedy, lin2021faster}.
%\citet{ye2021explicit} extended the results of \citet{rodomanov2021new} to the modified SR1 method.
%Above results \cite{rodomanov2021greedy, rodomanov2021new, rodomanov2021rates, lin2021faster,ye2021explicit} only need the assumption that the initial points should be in a small region near the unique solution, that is, $\norm{\vx_0-\vx^*}$ should be small enough.

\citet{jin2020non} provided a dimension-free superlinear convergence rate for Broyden-Fletcher-Goldfarb-Shanno (BFGS) method \cite{broyden1970convergence2, broyden1970convergence,fletcher1970new, goldfarb1970family, shanno1970conditioning}, and  Davidon-Fletcher-Powell (DFP) method \cite{davidon1991variable, fletcher1963rapidly} if the initial point and approximate Hessian are close enough to the solution and its true Hessian. 
The assumption used in \cite{jin2020non} is much similar to ours. 
Our work also requires the assumptions that both $\norm{\vx_0-\vx^*}$ and $\norm{\mB_0-\mJ(\vx^*)}$ should be small enough, but we mainly focus on nonlinear equations.
Furthermore, we choose $\vu_k$ greedily or randomly in the update of the approximate Jacobian which is much different from the one in \cite{jin2020non}. 

There are many variants of Broyden's method. Here we list some popular implementations and analogical versions of ours.
Broyden's ``bad'' method \cite{broyden1965class, martinez2000practical} directly applies the derivation to $\mH_k:=\mB_k^{-1}$ as below:
\[ \mH_{k+1} = \mH_k + \frac{\left(\vu_k-\mH_k\vy_k\right)\vy_k^\top}{\vy_k^\top\vy_k}. \]
In contrast, the ``bad'' method may perform not so well but on a small set of test problems \cite{al2014broyden}.
\citet{martinez2000practical} also mentioned a combined method that chose the ``good'' and the ``bad'' methods according to some conditions.

Moreover, some of the variants focus on handling the singularity of $\mB_k$ in Broyden's update (Line 5 in Algorithm \ref{algo:broyden}), such as choosing the minimum-norm
solution using a pseudoinverse minimizer $\mB_k^+$ instead of undefined $\mB_k^{-1}$, or a well-known approximation of the pseudoinverse $\left(\mB_k^\top\mB_k+\mu\mI_n\right)^{-1}\mB_k^\top$ with some $\mu>0$ to replace $\mB_k^+$ \cite{martinez2000practical}. 
Damped Broyden update \cite{al2014broyden, more1976global} adopts a modified update rule 
\[ \mB_{k+1} =  \mB_k + \theta_k \cdot  \frac{\left(\vy_k-\mB_k\vu_k\right)\vu_k^\top}{\vu_k^\top\vu_k}. \]
with alternative $\theta_k$ to guarantee conservative small-variation of $\left|\det\left(\mB_k\right)\right|$.
Furthermore, for some specific problems, researchers proposed structured methods \cite{avila1979update, hart1992solution}, such as using $ LU $ decomposition $\mB_0=\mL\mU$ throughout the calculations if $\mL$ is sparse and its inverse is easy to compute. Another method uses the decomposition $\mJ(\vx) = \mC(\vx) + \mD(\vx)$ if $ \mC(\vx) $ is easy to compute, and updates $\mD_k$ to approximate $\mD(\vx_k)$ from Broyden's update.
There also have other secant methods. For example, the Column-Updating method (COLUM) \cite{martinez1984quasi} is defined as
\[ \mB_{k+1} = \mB_k+\frac{\left(\vy_k-\mB_k\vu_k\right)\ve_{j_k}^\top}{\ve_{j_k}^\top\vu_k}, \]
where $|\ve_{j_k}^\top \vu_k| = \norm{\vu_k}_{\infty}$.
Such an update could reduce the computational cost of Broyden's method. Similarly, we also adopt coordinate vectors $\{\ve_1, \dots, \ve_n\}$ for updating the approximate matrix under a more aggressive strategy mentioned in Section \ref{sec:gre-ran-up}.
Furthermore, some work also employed an update rule as 
\[ \mB_{k+1} = \mB_k+\frac{\left(\vy_k-\mB_k\vu_k\right)\vv_{k}^\top}{\vv_{k}^\top\vu_k}, \text{ where } \vv_{k} \perp \vu_{k-1}, \]
which in some sense are close to the multipoint secant methods studied in \cite{burdakov1983stable, burdakov1986superlinear, schnabel1983quasi}.
These variants makes Broyden's method more practical and stable.
\section{Preliminaries}\label{sec:preliminaries}
\subsection{Notation}
We denote vectors by lowercase bold letters (e.g., $ \vu, \vx$), and matrices by capital bold letters (e.g., $ \mW = [w_{i j}] $).
We use $[n]:=\{1,\dots,n\}$ and $\mI_n$ is the $\sR^{n \times n}$ identity matrix.
Moreover, $\norm{\cdot}$ denotes the $\ell_2$-norm (standard Euclidean norm) for vectors, or induced $2$-norm (spectral norm) for a given matrix: $\norm{\mA} = \sup_{\norm{\vu}=1, \vu\in\sR^n}\norm{\mA\vu}$, and $\norm{\cdot}_F$ denotes the Frobenius norm of a given matrix: $\norm{\mA}_F^2 = \sum_{i=1}^{m} \sum_{j=1}^n a_{i j}^2$, where $\mA=[a_{i j}] \in \sR^{m\times n}$.
We also adopt $s_i(\mA)$ be its $ i $-th largest singular value ($i = 1, \dots, \min\{m,n\}$). 
When $\mA$ is nonsingular and $m=n$, we define its condition number as $\kappa(\mA):= s_1(\mA)/s_n(\mA)$. 
For two real symmetric matrices $\mA, \mB \in \sR^{n \times n}$, we denote $\mA\succeq\mB$ (or $\mB \preceq \mA$) if $\mA-\mB$ is a positive semidefinite matrix.
We use the standard $O(\cdot), \Omega(\cdot)$ and $\Theta(\cdot)$ notation to hide universal constant factors.
% Finally, we recall the definition of the rate of convergence used in this paper.
% \begin{defn}[Linear/Superlinear convergence]
% 	Suppose a sequence $\{\bm{x}_k\}$ converges to $\bm{x}_{\infty}$ that satisfies
% 	\[ \lim_{k\to \infty }\frac{\norm{\bm{x}_{k+1}-\bm{x}_{\infty}}}{\norm{\bm{x}_k-\bm{x}_{\infty}}} = q\in[0, 1). \]
% 	We say the sequence $\{\bm{x}_k\}$ converges superlinearly if $q=0$,
% 	linearly if $q\in(0, 1)$.
% \end{defn}

\subsection{Two Assumptions}

To analyze the convergence properties of our novel algorithms,  we first list two standard assumptions in convergence analysis of quasi-Newton methods.

\begin{assume}\label{ass:lisp}
	The Jacobian $\mJ(\vx)$ is Lipschitz continuous related to $\vx^*$ with
	parameter $M$ under induced $2$-norm, i.e., 
	\begin{equation}\label{eq:lisp}
		\norm{\mJ(\vx)-\mJ(\vx^*)} \leq M \norm{ \vx-\vx^*}.
	\end{equation}
\end{assume}

Such assumption also appears in \cite{jin2020non}, and the condition in
Assumption \ref{ass:lisp} only requires Lipschitz continuity in the direction of the optimal solution $\vx^*$. 
Indeed, that condition is weaker than general Lipschitz continuity for any two points. 
It is also weaker than the strong self-concordance recently proposed by \citet{rodomanov2021greedy, rodomanov2021new,rodomanov2021rates}.

In addition, we must show that the sequence $ \{\mB_k\}  $(or $ \{\mH_k\} $) generated by our variants  of Broyden's updates (Algorithm~\ref{algo:broyden-greedy-random}) exist and remain nonsingular shown in Assumption \ref{ass:optimal} below.

\begin{assume}\label{ass:optimal}
	The Broyden sequence $ ({\vx_k}, {\mH_k}) $ in Algorithm \ref{algo:broyden-greedy-random} is well-defined, that is $\mH_k$ is nonsingular for all
	$k$, and $ \vx_{k+1} $ and $ \mH_{k+1} $ can be well-defined by computing from $ \vx_{k} $ and $ \mH_{k}$. Moreover, the solution $\vx^*$ of Eq.~\eqref{obj} is nondegenerate, i.e., $\mJ_{*}:=\mJ(\vx^*)$ is nonsingular or $\mJ_{*}^{-1}$ exists.
\end{assume}

Assumption \ref{ass:optimal} also appears in \cite{dennis1996numerical, nocedal2006numerical}.
Additionally, along with the classical result on quadratic convergence of Newton’s methods, we need some sort of regularity conditions for the objective Jacobian following Assumption \ref{ass:lisp}.

\subsection{Broyden's Update}%\label{sec:broyden}
Before starting our theoretical results, we briefly review Broyden's update. We focus on solving a nonlinear equation:
\begin{equation*}
\mF(x) = \bm{0},
\end{equation*}
where $\mF: \sR^n\to \sR^n$ is differentiable. 
We denote its Jacobian as $\mJ(\vx)\in\sR^{n \times n}$. 
Similar to quasi-Newton methods, we need iteratively update the approximate Jacobian $\{\mB_k\}$ to approach the true Jacobian matrix, while updating the parameters $\{\vx_k\}$ as:
\begin{equation*}%\label{eq:x-update}
    \vx_{k+1} = \vx_k-\mB_k^{-1}\mF(\vx_k).
\end{equation*}

In several quasi-Newton methods, the alternative matrix $\mB_k$ need satisfy the secant condition, that is 
\begin{equation}\label{eq:secant}
\mB_k\vu_k=\vy_k, \mbox{ with }\vu_k = \vx_{k+1}-\vx_k,\; \vy_{k+1}=\mF(\vx_{k+1})-\mF(\vx_k).
\end{equation} 
However, we could not obtain a definite solution with only Eq.~\eqref{eq:secant} because it is an underdetermined linear system with an infinite number of solutions. 
Previous works adopt several different extra conditions to give a reasonable solution. 
The most common concern chooses the Jacobian approximation $\mB_{k+1}$ sequence is as close as possible to $\mB_{k}$ under a measure. 
The Broyden's update \cite{broyden1965class, kelley1995iterative} defined as follows 
\begin{equation}\label{eq:B-update-good}
	\mB_{k+1} =  \mB_k + \frac{\left(\vy_k-\mB_k\vu_k\right)\vu_k^\top}{\vu_k^\top\vu_k},
\end{equation}
is obtained by solving the following problem:
\begin{equation*}%\label{eq:good-update-form}
\min_{\mB} \norm{\mB-\mB_k}_F, \ s.t., \mB\vu_k=\vy_k.
\end{equation*} 
% The detailed derivation of the Broyden's method can be found in Appendix~\ref{app:}.
% However, we may encounter singular $\mB_{k+1}$ in practice, which leads to variants mentioned in related work. 
% Here we adopt the damped Broyden's method \cite{al2014broyden}, i.e.,
% \begin{align}
% \label{eq:db_up}
%   \mB_{k+1} =  \mB_k + \frac{\left(\vy_k-\mB_k\vu_k\right)\vu_k^\top}{\vu_k^\top\vu_k}.
% \end{align}
Additionally, for the practical implementation, researchers would adopt Sherman-Morrison-Woodbury formula \cite{sherman1950adjustment} for updating the inverse of matrix $\mH_k := \mB_{k}^{-1}$ as follows:
\begin{equation*}%\label{eq:H-update-good}
\mH_{k+1}=\mH_k-\frac{\left(\mH_k \vy_k-\vu_k\right) \vu_k^\top \mH_k}{\vu_k^\top\mH_k\vy_k}.
\end{equation*}
For better understanding the update rules of Broyden's methods, we turn to a simple linear objective $\mF(\vx) =\mA\vx-\vb$ for an explanation. 
When applied to the linear objective, we could simplify the update rules using the fact $\vy=\mA\vu$, which leads to the definition of Broyden's update as below.
\begin{defn}[Broyden's Update]\label{defn:broyd}
	Letting $\mB, \mA\in\sR^{n \times n}, \vu\in\sR^n$, we define
    \[ \text{Broyd} \left(\mB, \mA, \vu \right) := \mB +  \frac{\left(\mA-\mB\right)\vu\vu^\top}{\vu^\top\vu}. \]
\end{defn}
Then, for the simple linear objective $\mF(\vx) =\mA\vx-\vb$, Broyden's update can be written as $\mB_{k+1}=\text{Broyd} \left(\mB_k, \mA, \vu_k\right)$. 
We list the following well-known properties of Broyden's update:
\begin{prop}\label{prop:Broyden}
	Let $\mA, \mB \in\sR^{n\times n}, \vu\in\sR^n$ and $\mA$ is nonsingular. Denote $\mB_+ = \text{Broyd} \left(\mB, \mA, \vu\right)$ and set $\overline{\mB} := \mC\left(\mB-\mA\right)$, $\overline{\mB}_+ := \mC\left(\mB_+-\mA\right)$ for any nonsingular matrix $\mC \in \sR^{n \times n}$, then we have
	\begin{equation}\label{bro-up-prop}
	\overline{\mB}_+\overline{\mB}_+^\top = \overline{\mB}\overline{\mB}^\top-  \frac{\overline{\mB}\vu\vu^\top\overline{\mB}^\top}{\vu^\top\vu}.
	\end{equation}
	Therefore, we also have
	\begin{equation}\label{bro-up-prop-ex}
	\norm{\overline{\mB}_+}_F \leq \norm{\overline{\mB}}_F, \norm{\overline{\mB}_+} \leq \norm{\overline{\mB}}.
	\end{equation}
\end{prop}

\begin{remark}
	Proposition \ref{prop:Broyden} shows the monotonicity of Broyden's update under specific measures, e.g. $\norm{\cdot}, \norm{\cdot}_F$, which provides controllable sequence $\{\mB_n\}$ under iterative Broyden's update in Eq.~\eqref{eq:B-update-good}.
	These properties motivate our proofs in the subsequent sections when the initial conditions (of $\mB_0$) are benign.
\end{remark}
\section{Greedy and Random Broyden's Method}\label{sec:gre-ran-method}

In this section, we will first give a detailed algorithmic description and show how our greedy and random Broyden's methods find the solution of the objective function.
Then, we will provide the intuition behind our novel algorithms and give the reason why the approximate Jacobians in these two algorithms can converge to the exact one.
Next, we will give the main convergence properties of our two algorithms.
Finally, we compare the convergence rates of our algorithms with the original Broyden's method and show that our algorithms achieve faster asymptotic convergence rates.

\subsection{Algorithm Description}

The main difference among the greedy, random and classical Broyden's methods lies in the way to choosing $\vu_k$.
For the greedy Broyden's method, we choose $\vu_k$ by maximizing the following problem
\begin{equation*}
    \vu_k = \mathop{\arg\max}_{\vu \in \{\ve_1,\dots,\ve_n\}}\norm{\mB_k \vu - \mJ(\vx_{k+1})\vu},
\end{equation*}
where $\mB_k$ is the current approximate Jacobian and $\mJ(\vx_{k+1})$ is the exact Jacobian at $\vx_{k+1}$.
Accordingly, we can obtain that
\begin{equation}
\label{eq:y_k}
    \vy_k = \mJ(\vx_{k+1})\vu_k.
\end{equation}
Using the above $\vu_k$ and $\vy_k$, we can update the approximate Jacobian by Eq.~\eqref{eq:B-update-good}. 
Combining the algorithmic procedure of the classical Broyden's method, we can obtain the update rule of the greedy Broyden' method as follows:
\begin{equation}
\label{eq:x-up}
    \left\{
    \begin{aligned}
    \vx_{k+1} =& \vx_k - \mB_k^{-1} \mF(\vx_k)\\
    \vu_k =& \mathop{\arg\max}_{\vu \in \{\ve_1,\dots,\ve_n\}}\norm{\mB_k \vu - \mJ(\vx_{k+1})\vu}\\
    \vy_k =& \mJ(\vx_{k+1}) \vu_k
    \\
    \mB_{k+1} =& \mB_k + \frac{\left(\vy_k - \mB_k\vu_k\right)\vu_k^\top}{\vu_k^\top\vu_k}.
    \end{aligned}
    \right.
\end{equation}

For the random Broyden's method, we will choose $\vu_k$ randomly from distribution $\vu_k \sim \mathrm{Unif}\left(\left\{\ve_1, \dots, \ve_n\right\}\right)$.
Note that, $\vu_k$ can be chosen randomly from a broad class of distribution other than $\mathrm{Unif}\left(\left\{\ve_1, \dots, \ve_n\right\}\right)$ just as shown in the next subsection.
$\vy_k$ can be computed as Eq.~\eqref{eq:y_k} and the approximate Jacobian updates as Eq.~\eqref{eq:B-update-good} using obtained $\vu_k$ and $\vy_k$.
The whole algorithmic procedure of the random Broyden's method is the same as the one of the greedy Broyden's method except the way to obtain $\vu_k$.
Thus, the update rule of the random Broyden's method can be summarized by Eq.~\eqref{eq:x-up} but with  $\vu_k \sim \mathrm{Unif}\left(\left\{\ve_1, \dots, \ve_n\right\}\right)$.

The detailed algorithmic description of the greedy and random Broyden's methods are listed in Algorithm~\ref{algo:broyden-greedy-random}.

\begin{algorithm}[t]
	\caption{Greedy or Random Broyden's Method.}
	\begin{algorithmic}[1]
		\STATE Initialization: set $\mB_0, \vx_0$.
		\FOR{$ k \geq 0 $}
		\STATE Update $\bm{x}_{k+1} = \bm{x}_{k} - \mB_k^{-1} \mF(\vx_k)$.
		\STATE Coose $\bm{u}_k$ from \\
		1) \textit{greedy update}: $ \vu_k = \mathop{\arg\max}_{\vu \in \{\ve_1, \dots, \ve_n\}} \norm{\mB_k\vu-\mJ(\vx_{k+1})\vu}$, or \\
		2) \textit{random update}: $\vu_k \sim \mathrm{Unif}\left(\left\{\ve_1, \dots, \ve_n\right\}\right)$.
		%or $\vu_k \sim \mathrm{Unif}(\mathcal{S}^{n-1})$, or $ \vu_k \sim \mathcal{N}(\bm{0}, \mI_n)$.	
		\STATE Compute $\mB_{k+1} = \mB_k +  \frac{\left(\mJ(\vx_{k+1}) \vu_k - \mB_k\vu_k\right)\vu_k^\top}{\vu_k^\top\vu_k}$.
		\ENDFOR
	\end{algorithmic}
	\label{algo:broyden-greedy-random}
\end{algorithm}

\subsection{Motivation}
\label{sec:gre-ran-up}
It is well-known that the approximate Hessian of the quasi-Newton methods for unconstrained optimization will not converge to the true Hessian.
Similarly, quasi-Newton methods for solving nonlinear equations will not converge to the true Jacobian either.
Many works have provided several counterexamples that the approximate Hessian or Jacobian will not converge to the exact ones \cite{nocedal2006numerical,ren1983convergence,dennis1996numerical}.
In contrast to the quasi-Newton methods, Newton's method takes advantage of the current exact Jacobian in each step and achieves the quadratic convergence rate.
Such concern leads us to adjust the update rule to make $\{\mB_k\} $ converge to $\mJ(\vx^*)$ as close as possible.

Proposition \ref{prop:Broyden} shows, the classical Broyden's method only guarantees the distance between the approximate Jacobian and the exact one is non-increasing, but can not provide  convergence guarantee. 
The reason behind this phenomenon lies in the way of choosing $\vu_k$  shown in Eq~\eqref{eq:secant} and it is the reason why the approximate Jacobian of classical Broyden's method can not converge to the exact one. 
On the other hand, for Broyden's method, it is free to choose $\vu_k$.
Thus, we try to design novel Broyden's methods  by choosing proper $\vu_k$ such that the approximate Jacobian can converge to the true one iteratively. 
Specifically, we use the simple linear objective $\mF(\vx) = \mA\vx -\vb$ as an example and try to minimize the distance between $\mB_k$ and $\mA$ for the Frobenius norm.
For brevity, we denote $\mB_+ = \text{Broyd} \left(\mB, \mA, \vu\right)$ following Definition \ref{defn:broyd} and $\mB$ is the current approximate Jacobian.
One possible starting point is obtaining the greedy improvement under measure $\norm{\cdot}_F$ in each step from Proposition \ref{prop:Broyden}. 
That is, we choose the best update direction $\vu$ from a certain candidate set $\gI$:
\begin{equation*}%\label{eq:greedy-u-framework}
	\tilde{\vu} {=} \mathop{\arg\max}_{\vu \in \gI} \left(\norm{\mC\left(\mB{-}\mA\right)}_F^2 {-} \norm{\mC\left(\mB_+{-}\mA\right)}_F^2\right) {=} \mathop{\arg\max}_{\vu \in \gI} \frac{\vu^\top\left(\mB{-}\mA\right)^\top\mC^\top\mC\left(\mB{-}\mA\right)\vu}{\vu^\top\vu}.
\end{equation*}
At the same time, the way to find proper $\vu$ should take the computation cost into consideration.
The classical Broyden's method commonly needs $O(n^2)$ computation complexity for each iteration. Hence, we require that the computation cost of the greedy step for selecting $\vu$ is no larger than $O(n^2)$. 
Accordingly, we propose $\mC = \mI_n$ and $\gI = \{\ve_1, \dots, \ve_n\} $, that is, 
\begin{equation}\label{eq:greedy-u}
	(Greedy) \ \tilde{\vu} = \mathop{\arg\max}_{\vu \in \{\ve_1, \dots, \ve_n\}} \frac{\vu^\top\left(\mB-\mA\right)^\top\left(\mB-\mA\right)\vu}{\vu^\top\vu} = \mathop{\arg\max}_{k \in [n]} \norm{\left(\mB-\mA\right)\ve_k}^2.
\end{equation}
It is easy to check that the greedy step of finding proper $\tilde{\vu}$ takes still $O(n^2)$ computation cost.
By the fact that 
\begin{align*}
\label{eq:greed_u}
    \norm{(\mB-\mA)\tilde{\vu}} = \max_{k\in[n]}\norm{(\mB-\mA)\ve_k}^2 \ge \frac{1}{n}\sum_{k=1}^n\norm{(\mB-\mA)\ve_k}^2 = \frac{1}{n}\norm{\mB - \mA}^2_F.
\end{align*}
Combining with Eq.~\eqref{bro-up-prop}, we can obtain that
\begin{align*}
    \norm{\mB_+ - \mA}_F^2 
    \le
    \left(1 - \frac{1}{n}\right) \norm{\mB - \mA}_F^2,
\end{align*}
that is, by selecting $\vu$ greedily with respect to the Frobenius norm, the approximate Jacobian can converge to the true one $\mA$ linearly with rate $1 - \frac{1}{n}$.
This is the total intuition to design our greedy Broyden's method which guarantees that the approximate Jacobian will converge the exact one.
At the same, if the approximate Jacobian $\mB_{k}$ is close to $\mA$, then  the update 
\begin{equation}
\label{eq:x_up_1}
\vx_{k+1} = \vx_{k} - \mB_{k}^{-1}\mF(\vx_{k})
\end{equation}
can obtain a fast converge rate.
Since $\mB_k$ will converge to $\mA$ linearly, i.e., $\mB_k$ is getting closer and closer to $\mA$, the convergence rate of update \eqref{eq:x_up_1} become faster and faster as $k$ increases.
Thus, the greedy Broyden method can achieve superlinear convergence.

Other than greedily selecting $\vu_k$ which requires extra computation, we may wonder whether we can find $\vu$ randomly such that a similar result as Eq.~\eqref{eq:greedy-u} also holds in expectation.
Fortunately, such  $\vu$'s exist and only need to satisfy that 
\begin{equation}\label{eq:random-u}
(Random) \ \ \vu \sim\gD, \ \E_{\vu\sim \gD}\frac{\vu}{\norm{\vu}} \cdot \frac{\vu^\top}{\norm{\vu}} \propto \mI_n.
\end{equation}
The above condition requires normalized $\vu$ has isotropic second-order moment.
There are lots of examples satisfying Eq.~\eqref{eq:random-u} such as $\mathrm{Unif}\left(\left\{\ve_1, \dots, \ve_n\right\}\right)$, $\mathrm{Unif}(\mathcal{S}^{n-1})$ and $\mathcal{N}(\bm{0}, \mI_n)$.

We summarize how the greedily selected $\vu$ and the randomly selected $\vu$ help to approximate the true Jacobian in the following lemma. 

\begin{lemma}\label{lemma:b-progress}
	Let $\mA, \mB \in\sR^{n\times n}, \vu\in\sR^n$ and $\mA$ be nonsingular. Denoting $\mB_+ = \text{Broyd} \left(\mB, \mA, \vu\right)$ with $\vu$ selected from the greedy update in Eq.~\eqref{eq:greedy-u}, then we have
	\begin{equation*}%\label{greedy-up}
		\norm{\mB_+-\mA}_F^2 \leq \left(1-\frac{1}{n}\right) \norm{\mB-\mA}_F^2.
	\end{equation*}
	Moreover, if choosing $\vu$ based on the random update in Eq.~\eqref{eq:random-u}, then  it holds that
	\begin{equation*}%\label{random-up}
		\E_{\vu} \norm{\mB_+-\mA}_F^2 = \left(1-\frac{1}{n}\right) \norm{\mB-\mA}_F^2.
	\end{equation*}
\end{lemma}

\begin{remark}
	The classical Broyden's update could not provide such an explicit decreasing bound with respect to $\norm{\cdot}_F$, since there is no requirement of the direction $\vu$. After using greedy update or random update, we can guarantee that the approximate Jacobian converges linearly to the exact Jacobian matrix $\mA$.
\end{remark}

Lemma~\ref{lemma:b-progress} also provides the intuition  of the Jacobian convergence for general nonlinear equations since a general nonlinear can be approximated by a linear function $\mF(\vx) = \mF(\vx^*) + \mJ(\vx^*)(\vx - \vx^*)$ which is the first-order Taylor's expansion around one solution $\vx^*$. 
In the following subsection, we will give a detailed description that how the approximate Jacobians in greedy and random Broyden's methods converge to the exact ones for the general nonlinear equation.
Accordingly, we also show how the approximate Jacobian helps greedy and random Broyden's methods to achieve superlinear convergence rates.

\subsection{Main Results}

First, we give a theorem that sketches the convergence properties of our greedy and random Broyden's methods. 
\begin{thm}\label{thm:super}
	Suppose $n\geq 2$, the objective function $\mF(\cdot)$ satisfy Assumption~\ref{ass:lisp} and sequences $\{\vx_k\}$ and $\{\mH_k\}$ generated by Algorithm~\ref{algo:broyden-greedy-random} satisfy Assumptions~\ref{ass:optimal}. We denote $\|\mJ_{*}^{-1}\| = c$. If the initals $\vx_0$ and $\mB_0$ satisfy
	\begin{equation}\label{eq:init}
	48 \sqrt{n}c M \norm{\vx_0 - \vx^*} + c \norm{\mB_0 - \mJ(\vx_0)} \leq \frac{1}{3},
	\end{equation}
	then we have for all $k\geq 0$, 
	\begin{equation}\label{eq:sigma-linear}
		\E \left[ c\norm{\mB_k - \mJ(\vx_k)}\right] \leq e \left(1-\frac{1}{n}\right)^{k/2},
	\end{equation}
	and 
	\begin{equation}\label{eq:super-bound}
		\E \left[ \frac{\norm{\vx_{k+1}-\vx^*}}{\norm{\vx_{k}-\vx^*}}\right] \leq e \left(1-\frac{1}{n}\right)^{k/2}.
	\end{equation}
	The expectation considers randomness of the direction $\vu_0, \dots, \vu_{k}$, and when applied to the greedy method, we can view it with no randomness for the same notation. 
\end{thm}
Theorem \ref{thm:super} tells us two important results.
First, Eq.~\eqref{eq:sigma-linear} shows that the approximate Jacobian $\mB_k$ will converge to the exact one linearly.
Second, Eq.~\eqref{eq:super-bound} shows that $\{\vx_k\}$ superlinearly converges to $\vx^*$ since
\[ \lim_{k \to \infty} \frac{\norm{\vx_{k+1}-\vx^*}}{\norm{\vx_{k}-\vx^*}} \stackrel{\eqref{eq:super-bound}}{\leq} \lim_{k\to \infty} e \left(1-\frac{1}{n}\right)^{k/2} =0. \]

Based on Theorem~\ref{thm:super}, we are going to provide the explicit convergence rates of greedy and random Broyden's methods in the following two theorems.
\begin{thm}
	\label{thm:greedy}
Suppose $n\geq 2$, the objective function $\mF(\cdot)$ and initials  $\vx_0$ and $\mB_0$ satisfy the properties described in Theorem~\ref{thm:super}. Sequences $\{\vx_k\}$ and $\{\mH_k\}$ generated by Algorithm~\ref{algo:broyden-greedy-random} with greedy update satisfy Assumptions~\ref{ass:optimal}.  
Then, for the greedy Broyden's method, it holds that
\begin{align}
	\label{eq:super-p}
	\norm{\vx_k - \vx^*} 
	\le&
	e^k \left( 1 -\frac{1}{n} \right)^{k(k-1)/4} \cdot \norm{\vx_0 - \vx^*}, \mbox{ for all } k\ge 1\\
	\norm{\mB_k - \mJ_*}_F \leq& 2c^{-1}e\left(1-\frac{1}{n}\right)^{k/2}, \text{ when } k \geq 4n+3. \nonumber
\end{align}
\end{thm}

\begin{thm}
\label{thm:random}
	Suppose $n\geq 2$, the objective function $\mF(\cdot)$ and initials  $\vx_0$ and $\mB_0$ satisfy the properties described in Theorem~\ref{thm:super}. Sequences $\{\vx_k\}$ and $\{\mH_k\}$ generated by Algorithm~\ref{algo:broyden-greedy-random} with random update satisfy Assumptions~\ref{ass:optimal}.  
	Then, for the random Broyden's method, it holds with probability at least $1 - \delta$ with $0<\delta<1$ that 
	\begin{align}
		\norm{\vx_k - \vx^*} 
		\le&
		\left(\frac{4n^2e}{\delta}\right)^k\left(1-\frac{1}{n+1}\right)^{k(k-1)/4} \cdot \norm{\vx_0 - \vx^*}, \mbox{ for all } k\ge 1 \label{eq:super_rand}\\
		\norm{\mB_k - \mJ_*}_F \leq&  \frac{8n^2c^{-1}e}{\delta}\left(1-\frac{1}{n+1}\right)^{k/2}, \text{ when } k \geq 4(n+1)\ln\frac{4n^2e}{\delta} + 3. \nonumber
	\end{align}
\end{thm}

Above two theorems provide the explicit superlinear convergence rates of greedy and random Broyden's methods with respect to the residual $\norm{\vx_k - \vx^*}$ and show that the approximate Jacobian $\mB_k$ will converge to $\mJ_{*}$ linearly.
Furthermore, the results in Theorem~\ref{thm:random} hold with constant probability. For example, with probability at least $0.9$, it holds that
\begin{align*}
\norm{\vx_k - \vx^*} 
\le
\left(40n^2 e\right)^k\left(1-\frac{1}{n+1}\right)^{k(k-1)/4} \cdot \norm{\vx_0 - \vx^*}.
\end{align*}

On the other hand, the superlinear convergence rates in the above theorems are meaningful only when $k$ is larger than a threshold.
This can be obtained from Eq.~\eqref{eq:super-bound} that $e\left(1 - \frac{1}{n}\right)^{k/2}$ should be less than $1$, i.e., the superlinear convergence rates are meaningful when $k \ge 2n$.
This phenomenon is the same as the greedy and random quasi-Newton methods for minimizing convex functions \cite{lin2021faster,rodomanov2021greedy,rodomanov2021rates,rodomanov2021new}.

\subsection{Comparison of Rates}

First, we provide a comparison between the greedy and random Broyden's methods.
By Eq.~\eqref{eq:super-p} and \eqref{eq:super_rand}, we can observe that the greedy Broyden's method is a little faster than the random Broyden's method.
However, these two algorithms share almost the same asymptotic convergence rate since $\left(1-1/n\right)^{k(k-1)/4}$ and $\left(1 - 1/(n+1)\right)^{k(k-1)/4}$ dominate the convergence properties described in Eq.~\eqref{eq:super-p} and \eqref{eq:super_rand}, respectively.

Next, we will compare the convergence rate of the greedy Broyden's method with the original Broyden's method obtained recently in \cite{lin2021explicit}.
\citet{lin2021explicit} gave an explicit local superlinear convergence rate of the original Broyden's method as follows:
\begin{equation}\label{eq:lambda-ori}
	\lambda_{k} \leq \left[\frac{1}{k}\right]^{k/2}\lambda_0, \text{ with } \lambda_k = \norm{\mJ_*^{-1}\left(\mF(\vx_k)-\mF(\vx^*)\right)},
\end{equation}
if 
\begin{align}
	 48 c M \norm{\vx_0 - \vx^*} + c\norm{\mB_0 - \mJ(\vx_0)} \leq \frac{1}{3}. \label{eq:cond}
\end{align}
Comparing Eq.~\eqref{eq:cond} with \eqref{eq:init}, we can conclude that our greedy and random Broyden's methods require a stronger initial condition.
However, our initial condition can reduce to one in \cite{lin2021explicit} if we assume that the Jacobian is $M$-Lipschitz continuous with respect to the Frobenius norm, that is, 
$\norm{\mJ(\vx) - \mJ(\vx^*)}_F \le M \norm{\vx - \vx^*}_F$.

It is hard to compare the convergence rates in Eq.~\eqref{eq:lambda-ori} and \eqref{eq:super-p}. 
We first transform Eq.~\eqref{eq:lambda-ori} with respect to $r_k := \norm{\vx_k - \vx^*}$.
It holds that
\begin{align*}
\norm{\mJ_*^{-1}\left[\mF(\vx_k)-\mF(\vx^*)\right]-\left(\vx_k-\vx^*\right)}
\le
\norm{\mJ_*^{-1}} \cdot \norm{\mF(\vx_k)-\mF(\vx^*)-\mJ_*\left(\vx_k-\vx^*\right)}
\stackrel{\eqref{eq:lemma-j2}}{\le} \frac{c M r_k}{2} \cdot r_k
\end{align*}
leading to
\[ \left(1 - \frac{c M r_k}{2}\right) r_k \leq \lambda_k \leq \left(1 + \frac{c M r_k}{2}\right) r_k. \]
The initialization in Eq.~\eqref{eq:lambda-ori} and the linear convergence of $\{r_k\}$ of Broyden's method \cite{lin2021explicit} shows that $c M r_k \leq \frac{1}{144}$. Thus it holds that $\lambda_k = \Theta(r_k)$, and we can obtain 
\begin{equation}
\label{eq:lam}
	 \norm{\vx_k - \vx^*} = r_{k} \leq 2\left(\frac{1}{k}\right)^{k/2}r_0 = 2\left(\frac{1}{k}\right)^{k/2}\cdot\norm{\vx_0 - \vx^*}.
\end{equation}
Now, we can compare the rate in Eq.~\eqref{eq:lam} with the one in Eq.~\eqref{eq:super-p}.
First, we have
\begin{equation*}
	\begin{aligned}
		\ln \left[2\left(\frac{1}{k}\right)^{k/2}\right] - \ln\left[e^{k-\frac{k(k-1)}{4n}} \right] &= \ln2 + \frac{k}{2}\ln\frac{1}{k} - k + \frac{k(k-1)}{4n} = \frac{k}{2} \left[\frac{k-1}{2n}-\ln k-2 \right] + \ln2.
	\end{aligned}
\end{equation*}
Since $t - \ln\left(2n t + 1\right)$ is increasing when $t\geq 1$, it holds that when $k \geq 8n \ln (ne) + 1$ and $n \geq 2$,
\begin{equation*}
	\begin{aligned}
		\frac{k-1}{2n} - \ln k - 2 \geq \ln \frac{n^4e^2}{8n \ln(ne) + 1} \geq \ln \frac{9n^2}{8n(n-1+1) + 1} \geq \ln \frac{9n^2}{8n^2+1}>0.
	\end{aligned}
\end{equation*}
where we use the inequality $0 \leq \ln t \leq t - 1, t \geq 1$.
Hence
\[ \ln \left[2\left(\frac{1}{k}\right)^{k/2}\right] - \ln\left[e^{k-\frac{k(k-1)}{4n}} \right] \geq \frac{k}{2} \ln \frac{9n^2}{8n^2+1} + \ln 2 \geq 0, \]
showing that
\[ 2\left(\frac{1}{k}\right)^{k/2} \geq e^{k-\frac{k(k-1)}{4n}}. \]
Therefore, when $k \geq 8n \ln (ne) + 1 = \Theta(n\ln n)$, the greedy Broyden's method could be faster than the original Broyden's method.

Since the random Broyden's method has almost the same asymptotic convergence rate to the greedy Broyden's method,
we can conclude that both greedy and random Broyden's methods can achieve faster convergence rates than the original Broyden's method asymptotically.

\section{Analysis of Greedy and Random Broyden's Method}
\label{sec:analysis}

In the last section, we provide the intuition behind the greedy and random Broyden's methods based on the simple linear objective function and the main results of this paper.
Now, we are going to give a detailed analysis of  the explicit superlinear convergence rates of greedy and random Broyden's methods for the general nonlinear equations.

%Now let's consider the convergence rate of Broyden's method. Though we use $\mH_k$ for practical operation, our theoretical analysis still focus on original $\{\mB_k\} $, which is easy to follow. 
First, for notation convenience, we introduce the measure for variables as
\begin{equation}\label{var-defn}
\vs_k := \vx_k-\vx^*, r_k := \|\vs_k\|.
\end{equation}
and for its Jacobian as
\begin{equation}\label{jac-measure}
	\sigma_k := \norm{\mB_k-\mJ_k}_F, \mJ_k :=\mJ(\vx_k). 
\end{equation}

Employing these measures, we can build up the relation between $\sigma_{k+1}$ and $\sigma_k$, as well as the relation between $r_{k+1}$ and $r_k$.

\begin{lemma}\label{lemma:sigma-up}
	Suppose the objective function $\mF(\cdot)$ satisfies Assumption~\ref{ass:lisp} and sequences $\{\vx_k\}$ and $\{\mB_k\}$ generated by Algorithm~\ref{algo:broyden-greedy-random} satisfy Assumption~\ref{ass:optimal}. 
	Then we have
	\begin{equation}\label{eq:sigma-update1}
		\sigma_{k+1}^2 \leq \sigma_k^2 + 2\sigma_k \sqrt{n}M\left(r_k+r_{k+1}\right)+n M^2\left(r_k+r_{k+1}\right)^2.
	\end{equation}
	More rigorously, we have
	\begin{equation}\label{eq:sigma-update2}
		\E_{\vu_{k}} \left[\sigma_{k+1}^2\right] \leq \left(1-\frac{1}{n}\right)\sigma_k^2 + 2\sigma_k \sqrt{n}M\left(r_k+r_{k+1}\right)+n M^2\left(r_k+r_{k+1}\right)^2.
	\end{equation}
	The expectation only considers the randomness of the direction $\vu_k$ (with fixed $\vu_0,\dots, \vu_{k-1}$), and when applied to the greedy method, we can view it with no randomness for the same notation.  
\end{lemma}
\begin{remark}
The results in Lemma~\ref{lemma:sigma-up} depend on the dimension $n$ which is commonly undesired. 
This dependency comes the relationship between the spectral norm and the Frobenius norm.
We can remove this dependence of $n$ by assuming that the Jacobian is $M$-Lipschitz continuous with respect to the Frobenius norm, that is, 
$\norm{\mJ(\vx) - \mJ(\vx^*)}_F \le M \norm{\vx - \vx^*}_F$.
\end{remark}

\begin{lemma}\label{lemma:r-up}
	Suppose the objective function $\mF(\cdot)$  and sequences $\{\vx_k\}$ and $\{\mB_k\}$ satisfy the properties described in Lemma~\ref{lemma:sigma-up}.  
	Then, we have
	\begin{equation}\label{eq:r-update}
	r_{k+1} \leq \frac{3cMr_k/2+c\sigma_k}{1-c\left(\sigma_k+Mr_k\right)}r_k, \text{ if } c\left(\sigma_k+Mr_k\right)<1, \text{ where } c = \|\mJ_{*}^{-1}\| > 0.
	\end{equation}
	
\end{lemma}

The above two lemmas show that $r_k$ and $\sigma_k$ are interrelated. 
However, if $\sigma_k$'s and $r_k$'s are small enough, then $\{r_{k}\} $ converges linearly by Lemma \ref{lemma:r-up}. Meanwhile, since sequence $\{r_{k}\} $ is of linear convergence, then $\{r_{k}\} $ is summable and leads to the boundedness of $\{\sigma_{k}\} $ by Lemma \ref{lemma:sigma-up}. Therefore, we could guarantee a linear convergence rate of $r_k$ under the condition on the benign initialization of $r_0$ and $\sigma_0$.
\begin{lemma}\label{lemma:base}
	Suppose $n\geq 2$, the objective function $\mF(\cdot)$ satisfy Assumption~\ref{ass:lisp} and sequences $\{\vx_k\}$ and $\{\mB_k\}$ generated by Algorithm~\ref{algo:broyden-greedy-random} satisfy Assumptions~\ref{ass:optimal}. We denote $\|\mJ_{*}^{-1}\| = c$. 
	Letting $q \in (0, 1)$, $r_0$ and $\sigma_0$ satisfy
	\begin{equation}\label{eq:init-cond}
	c\sigma_0 \leq \frac{q}{q+1}, \ \ \sqrt{n}cM r_0 \leq \frac{q(1-q)}{12}\left(\frac{q}{1+q}-c\sigma_0\right),
	\end{equation}
	then we have for all $k\geq 0$, 
	\begin{equation}\label{eq:sigma-bound}
	c^2\sigma_{k}^2 \leq c^2\sigma_0^2 + \frac{1+q}{1-q}\left(2c\sqrt{n}Mr_0+ nc^2M^2r_0^2\right) \leq \frac{q^2}{(1+q)^2},
	\end{equation}
	and 
	\begin{equation}\label{eq:linear-con}
	r_{k+1} \leq q r_k.
	\end{equation}
\end{lemma}

To better quantify the linear convergence rate of $r_k$, we define $q_m$ given the suitable $\sigma_0$ and $r_0$, 
\begin{equation*}%\label{eq:qm}
q_m := \min q, \ s.t. \ c\sigma_0 \leq \frac{q}{q+1}, \ \sqrt{n}c M r_0 \leq \frac{q(1-q)}{12} \left(\frac{q}{1+q}-c\sigma_0\right).
\end{equation*}
Such $q_m$ needs a pair of proper $\sigma_0$ and $ r_0$ to guarantee a nonempty feasible set, and $q_m$ determines the optimal linear convergence rate provided by Broyden's method from Lemma \ref{lemma:base}. 
Then we could obtain a more understandable result.

\begin{corollary}\label{coro:good}
	Suppose the objective function $\mF(\cdot)$ and sequences $\{\vx_k\}$ and $\{\mB_k\}$ satisfy the properties described in Lemma~\ref{lemma:base}. If $r_0$ and $\sigma_0$ satisfy
	\begin{equation}\label{eq:init-super}
	48 \sqrt{n} c M r_0 + c\sigma_0 \leq \frac{1}{3},
	\end{equation}
	then $q_m$ is well-defined and we have $q_m =  \Theta\left(c\sigma_0 + \sqrt{n^{1/2} c M r_0}\right)$.
\end{corollary}

Based on the above lemmas, we are going to prove the main results in Section~\ref{sec:gre-ran-method}.
First, we give the detailed proof of Theorem~\ref{thm:super}.

\begin{proof}[Proof of Theorem~\ref{thm:super}]
	The results in Lemma \ref{lemma:base} still hold since Corollary \ref{coro:good} shows that there exists such a $q$ that satisfies the requirement in Lemma~\ref{lemma:base} 
	Thus, combining with Lemma \ref{lemma:r-up}, we have
	\begin{equation}\label{eq:rk-up-again}
		r_{k+1} \stackrel{(\ref{eq:r-update})}{\leq} \frac{3cMr_k/2+c\sigma_k}{1-c\sigma_k-cMr_k}r_k, \text{ if } c\sigma_k+cMr_k<1.
	\end{equation}
	From Lemma~\ref{lemma:base}, we have $c\sigma_k+cMr_k \stackrel{(\ref{eq:cmr-si})}{\leq} \frac{q}{q+1}<1$, showing that Eq.~\eqref{eq:rk-up-again} holds and 
	\begin{equation}\label{eq:rk-up-super}
		r_{k+1} \leq (q+1)\left(3cMr_k/2+c\sigma_k\right)r_k.
	\end{equation}
Thus, we could derive the update of $\sigma_{k}$ more rigorously by Lemma~\ref{lemma:sigma-up} as follows
	\begin{equation}\label{eq:sigma-k-up2}
		\begin{aligned}
			\E_{\vu_{k}} \left[\sigma_{k+1}^2\right] 
			&\stackrel{(\ref{eq:sigma-update2})}{\leq} \left(1-\frac{1}{n}\right)\sigma_k^2 + 2\sigma_k \sqrt{n}M\left(r_k+r_{k+1}\right)+n M^2\left(r_k+r_{k+1}\right)^2 \\
			&\leq \left(1-\frac{1}{n}\right)\sigma_k^2 + 2\sigma_k \sqrt{n}M\left(r_k+r_{k+1}\right)+n M^2\left(r_k+r_{k+1}\right)^2 \\
			&\stackrel{(\ref{eq:linear-con})}{\leq} \left(1-\frac{1}{n}\right)\sigma_k^2 + 2\sigma_k (1+q) \sqrt{n}Mr_k + (1+q)^2nM^2r_k^2 \\
			& \leq \left(1-\frac{1}{n}\right)\left[\sigma_k + \frac{1}{1-\frac{1}{n}}(1+q) \sqrt{n}Mr_k \right]^2 \\
			& \leq \left(1-\frac{1}{n}\right)\left[\sigma_k + 2(1+q) \sqrt{n}Mr_k \right]^2,
		\end{aligned}	
	\end{equation}
	where the last inequality uses the assumption that 
	\begin{equation}\label{eq:n}
		\frac{1}{n} \leq \frac{1}{2}, n\geq 2.
	\end{equation}
	Hence we obtain the update of $\sigma_k$ (in expectation):
	\begin{equation}\label{eq:sigma-k-up2-super}
		\begin{aligned}
			\E_{\vu_{k}} \left[\sigma_{k+1}\right] \leq \sqrt{\E_{\vu_{k}} \left[\sigma_{k+1}^2\right]} &\stackrel{(\ref{eq:sigma-k-up2})}{\leq} \sqrt{\left(1-\frac{1}{n}\right)}\left(\sigma_k + 2(1+q) \sqrt{n}Mr_k \right).
		\end{aligned}	
	\end{equation}
	Now we define $R_k = \sigma_{k}+2(1+q)\sqrt{n}Mr_{k}$.
	Then, we have
	\begin{equation*}
		\begin{aligned}
			\E_{\vu_{k}}\left[ R_{k+1}\right] &\stackrel{\eqref{eq:sigma-k-up2-super}\eqref{eq:rk-up-again}}{\le} \sqrt{\left(1-\frac{1}{n}\right)}\left(\sigma_k + 2(1+q) \sqrt{n}Mr_k \right) + \left(3Mr_k/2+\sigma_k\right) \cdot 2c(1+q)^2\sqrt{n}Mr_k \\
			&\stackrel{(\ref{eq:n})}{\leq} \sqrt{\left(1-\frac{1}{n}\right)}R_k + \sqrt{2\left(1-\frac{1}{n}\right)} \left(3Mr_k/2+\sigma_k\right) \cdot 2c(1+q)^2\sqrt{n}Mr_k \\
			&\leq \sqrt{\left(1-\frac{1}{n}\right)}R_k + \sqrt{\left(1-\frac{1}{n}\right)} R_k \cdot 2\sqrt{2}c(1+q)^2\sqrt{n}Mr_k \\
			&= \sqrt{\left(1-\frac{1}{n}\right)}R_k \left(1+ 2\sqrt{2}c(1+q)^2\sqrt{n}Mr_k\right) \\
			&\leq \sqrt{\left(1-\frac{1}{n}\right)}R_k \exp{\left(2\sqrt{2}c(1+q)^2\sqrt{n}Mr_k\right)} \\
			&\stackrel{(\ref{eq:linear-con})}{\leq} \sqrt{\left(1-\frac{1}{n}\right)}R_k \exp{\left(2\sqrt{2}c(1+q)^2\sqrt{n}Mq^kr_0\right)}.
		\end{aligned}
	\end{equation*}	
	Now we take expectation for all $\vu_i$'s in both sides and telescope from $k$ to $0$, we get
	\begin{equation}\label{eq:RkR0}
		\begin{aligned}
			\E \left[ R_{k}\right] &\leq \left(1-\frac{1}{n}\right)^{k/2}R_0 \exp{\left(2\sqrt{2}c(1+q)^2\sqrt{n}M\sum_{i=0}^{k}q^ir_0\right)} \\
			&\leq \left(1-\frac{1}{n}\right)^{k/2}R_0 \exp{\left(2\sqrt{2}c(1+q)^2\sqrt{n}Mr_0/(1-q)\right)} \\
			&\stackrel{(\ref{eq:init})}{\leq} \left(1-\frac{1}{n}\right)^{k/2}R_0 \exp\left\{2\sqrt{2} \cdot \frac{(1+q)^2}{1-q}\cdot \frac{q(1-q)}{12} \cdot \frac{q}{1+q}\right\} \\
			&= \left(1-\frac{1}{n}\right)^{k/2}R_0 \exp{\left(q^2\right)} \leq \left(1-\frac{1}{n}\right)^{k/2}R_0 e.
		\end{aligned}
	\end{equation}
	Moreover, 
	\begin{equation}\label{eq:R0}
		cR_0 = c\sigma_{0}+2(1+q)c\sqrt{n}Mr_{0} \stackrel{(\ref{eq:init})}{\leq} c\sigma_0+ \left(\frac{q}{q+1}-c\sigma_0\right) = \frac{q}{q+1}<1.
	\end{equation}
	Therefore, we obtain
	\begin{equation}\label{eq:Rk}
		\E \sigma_k \leq \E R_{k} \stackrel{\eqref{eq:RkR0}, \eqref{eq:R0}}{\leq} \frac{e}{c} \left(1-\frac{1}{n}\right)^{k/2}.
	\end{equation}
	Hence Eq. \eqref{eq:sigma-linear} is proved.
	Furthermore,  we could obtain 
	\begin{equation*}
		\E \left[\frac{r_{k+1}}{r_k}\right] 
		\stackrel{\eqref{eq:rk-up-super}}{\le} \E \left[(q+1)\left(3cMr_k/2+c\sigma_k\right)\right] \leq \E \left[cR_k\right] \stackrel{(\ref{eq:Rk})}{\leq} e \left(1-\frac{1}{n}\right)^{k/2},
	\end{equation*}
	which concludes the proof.
\end{proof}

Based on Theorem~\ref{thm:super} and above lemmas, we are going to prove Theorem~\ref{thm:greedy} as follows.
\begin{proof}[Proof of Theorem~\ref{thm:greedy}]
	First, Eq.~\eqref{eq:super-bound} holds without any randomness since $\vu_k$ is not a random vector.
	Thus, we have
	\begin{align*}
		\norm{\vx_k - \vx^*} =& \prod_{i=1}^k \frac{\norm{\vx_i - \vx^*}}{\norm{\vx_{i-1} - \vx^*} } \cdot \norm{\vx_0 - \vx^*}
		\stackrel{\eqref{eq:super-bound}}{\le} e^k \prod_{i=1}^k \left(1 - \frac{1}{n}\right)^{\frac{i-1}{2}}\cdot \norm{\vx_0 - \vx^*}
		\\
		=&e^k \left( 1 -\frac{1}{n} \right)^{\frac{k(k-1)}{4}} \cdot \norm{\vx_0 - \vx^*}.
	\end{align*}
	Furthermore, when $k \geq 4n+3$, we get 
	\[ \ln\left[e^{k-1}\left(1-\frac{1}{n}\right)^{k(k-3)/4}\right] \leq k-1-\frac{k(k-3)}{4n} \leq k-1-k \leq 0,\]
	showing that
	\begin{equation}\label{eq:q}
		e^k\left(1-\frac{1}{n}\right)^{k(k-1)/4} \leq e\left(1-\frac{1}{n}\right)^{k/2}.
	\end{equation}
	From Theorem \ref{thm:super}, we obtain
	\begin{equation*}
		\begin{aligned}
			c\norm{\mB_k - \mJ_*}_F &\leq c\left[\norm{\mB_k - \mJ_k}_F + \norm{\mJ_k - \mJ_*}_F \right]\stackrel{(\ref{jac-measure}, \ref{eq:lisp})}{\leq }  c \sigma_k + c\sqrt{n}Mr_k  \\
			&\stackrel{(\ref{eq:sigma-linear}, \ref{eq:super-bound})}{\leq} e\left(1-\frac{1}{n}\right)^{k/2} + e^k\left(1-\frac{1}{n}\right)^{k(k-1)/4} c\sqrt{n}Mr_0 \\
			&\stackrel{(\ref{eq:init})}{\leq } e\left(1-\frac{1}{n}\right)^{k/2} + e^k\left(1-\frac{1}{n}\right)^{k(k-1)/4}
			\stackrel{\eqref{eq:q}}{\leq} 2e\left(1-\frac{1}{n}\right)^{k/2}.
		\end{aligned}
	\end{equation*}
\end{proof}
Now we are going to prove Theorem~\ref{thm:random} as follows.
\begin{proof}[Proof of Theorem~\ref{thm:random}]
	We consider random variable $\mX_k=c\sigma_k, \forall k \geq 0$ or $\mX_k=r_{k+1}/r_k, \forall k \geq 0$ in the following derivation.
	Note that $\mX_k \geq 0$, using Markov's inequality, we have for any $\epsilon>0$, 
	\begin{equation}\label{eq:mi}
		\sP\left(\mX_k \geq \frac{e}{\epsilon} \left(1-\frac{1}{n}\right)^{k/2} \right) \leq \frac{\E \ \mX_k}{\frac{e}{\epsilon} \left(1-\frac{1}{n}\right)^{k/2}} \stackrel{\eqref{eq:sigma-linear} or \eqref{eq:super-bound}}{\leq} \epsilon.
	\end{equation}
	Choosing $\epsilon_k = \delta (1-q)q^k$ for some positive $q < 1$, then we have 
	\begin{equation*}
		\begin{aligned}
			\sP\left(\mX_k \geq \frac{e}{\epsilon_k} \left(1-\frac{1}{n}\right)^{k/2}, \exists \ k\in\sN\right) 
			\leq&  \sum_{k=0}^\infty \sP\left(\mX_k
			\geq \frac{e}{\epsilon_k}\left(1-\frac{1}{n}\right)^{k/2} \right) \\
			\stackrel{(\ref{eq:mi})}{\leq}&  \sum_{k=0}^\infty \epsilon_k = \sum_{k=0}^\infty \delta (1-q)q^k = \delta.
		\end{aligned}
	\end{equation*}
	Therefore, we obtain with probability $1-\delta$,
	\begin{equation*}%\label{eq:v2}
		\mX_k \leq \left(\frac{1-\frac{1}{n}}{q^2}\right)^{k/2} \cdot \frac{e}{(1-q)\delta}, \forall k\in\sN.
	\end{equation*}
	If we set $q = \sqrt{1 - \left(\frac{1}{n}\right)^2} < 1-\frac{1}{2n^2}$, we could obtain with probability $1-\delta$, for all $k \in \sN$,
	\begin{equation*}%\label{eq:v3}
		\mX_k \leq \frac{2n^2e}{\delta}\left(1+\frac{1}{n}\right)^{-k/2} = \frac{2n^2e}{\delta}\left(1-\frac{1}{n+1}\right)^{k/2}.
	\end{equation*}
That is, for all $k$, it holds with probability $1-\delta$ that
\begin{align}
	\frac{r_{k+1}}{r_k} \le \frac{4n^2e}{\delta}\left(1-\frac{1}{n+1}\right)^{k/2}, \label{eq:r_tel}
\end{align}
and
\begin{align}
	\label{eq:super-r}
	c\sigma_{k} \le \frac{4n^2e}{\delta}\left(1-\frac{1}{n+1}\right)^{k/2}. 
\end{align}
	By telescoping from $k$ to $0$, it holds with probability $1-\delta$ that
	\begin{align*}
	r_{k} = r_0 \cdot \prod_{i=1}^k \frac{r_i}{r_{i-1}} \stackrel{\eqref{eq:r_tel}}{\le} r_0 \cdot \left(\frac{4n^2e}{\delta}\right)^{k} \prod_{i=1}^k \left(1 - \frac{1}{n+1}\right)^{\frac{i-1}{2}}
	=  \left(\frac{4n^2e}{\delta}\right)^{k}\left(1-\frac{1}{n+1}\right)^{k(k-1)/4} r_0.
	\end{align*}
	
	Furthermore, when $k \geq 4(n+1)\ln\frac{4n^2e}{\delta} + 3$, we get 
	\begin{equation*}
		\ln \left[\left(\frac{4n^2e}{\delta}\right)^{k-1}\left(1-\frac{1}{n+1}\right)^{k(k-3)/4}\right] \leq (k-1)\ln\frac{4n^2e}{\delta} -\frac{k(k-3)}{4(n+1)} \leq 0.
	\end{equation*}
	Hence,
	\begin{equation}\label{eq:q2}
		\left(\frac{4n^2e}{\delta}\right)^k\left(1-\frac{1}{n+1}\right)^{k(k-1)/4} \leq \frac{4n^2e}{\delta}\left(1-\frac{1}{n+1}\right)^{k/2}.
	\end{equation}
	Using above equations, we have
	\begin{equation*}
		\begin{aligned}
			c\norm{\mB_k - \mJ_*}_F &\leq c\left[\norm{\mB_k - \mJ_k}_F + \norm{\mJ_k - \mJ_*}_F \right]\stackrel{(\ref{jac-measure}, \ref{eq:lisp})}{\leq } c \sigma_k + c\sqrt{n} M r_k  \\
			&\stackrel{(\ref{eq:super-p}, \ref{eq:super-r})}{\leq} \frac{4n^2e}{\delta}\left(1-\frac{1}{n+1}\right)^{k/2} + \left(\frac{4n^2e}{\delta}\right)^k\left(1-\frac{1}{n+1}\right)^{k(k-1)/4} c\sqrt{n}Mr_0 \\
			&\stackrel{(\ref{eq:init})}{\leq } \frac{4n^2e}{\delta}\left(1-\frac{1}{n+1}\right)^{k/2} + \left(\frac{4n^2e}{\delta}\right)^k\left(1-\frac{1}{n+1}\right)^{k(k-1)/4} \\
			&\stackrel{\eqref{eq:q2}}{\leq} \frac{8n^2e}{\delta}\left(1-\frac{1}{n+1}\right)^{(k-1)/4},
		\end{aligned}
	\end{equation*}
which concludes the proof.
\end{proof}

\section{Experiments}\label{sec:exp}
In this section, we conduct numerical experiments for greedy and random Broyden's methods.
We adopt two kinds of common initialization to compare greedy and random versions with the original Broyden's method.
\begin{itemize}
	\item \textbf{Good Initial}: The first initialization employs $\vx_0$ and $\mB_0$ to be close to a solution $\vx^*$ like our theory mentioned. Specifically, we employ the initial approximate Jacobian $\mB_0 = \mJ(\vx_0)$ which is common in practice \cite{nocedal2006numerical} for Broyden's method. 
	Moreover, we may use Newton's method to obtain a (nondegenerate) solution $\vx^*$ and then choose initial $\vx_0$ after a little perturbation, or adopt a random $\vx_0$ for comparison with the ``bad initials'' below.
	\item \textbf{Bad Initial}: The second initialization adopts a random $\vx_0$ and $\mB_0 = sI_n$ with alternative $s>0$. In general smooth and strongly convex problems, $s$ can be the smooth coefficient, i.e., the upper bound of the maximal eigenvalue of Hessians. %While in general nonlinear equations, we select $s$ randomly.
\end{itemize}

\paragraph{Regularized Log-Sum-Exp.}

\begin{figure}[t]
	\centering
	\hspace{-8pt}
	\begin{subfigure}[b]{0.5\textwidth}
		\includegraphics[width=1\linewidth]{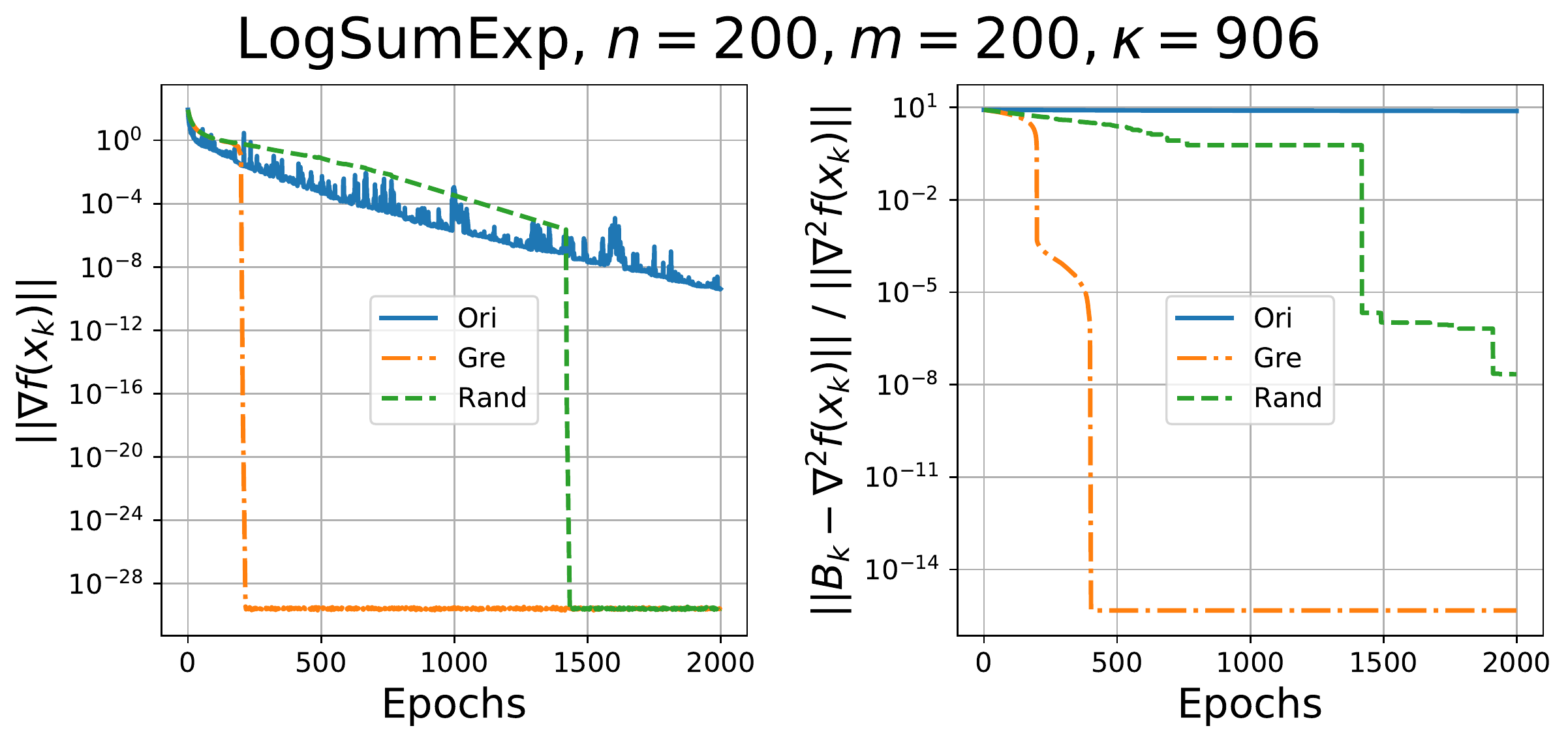}
		\caption{$\mB_0=L\mI_n$.} \label{fig:exp-gel}
	\end{subfigure}
	\hspace{-8pt}
	\begin{subfigure}[b]{0.5\textwidth}
		\includegraphics[width=1\linewidth]{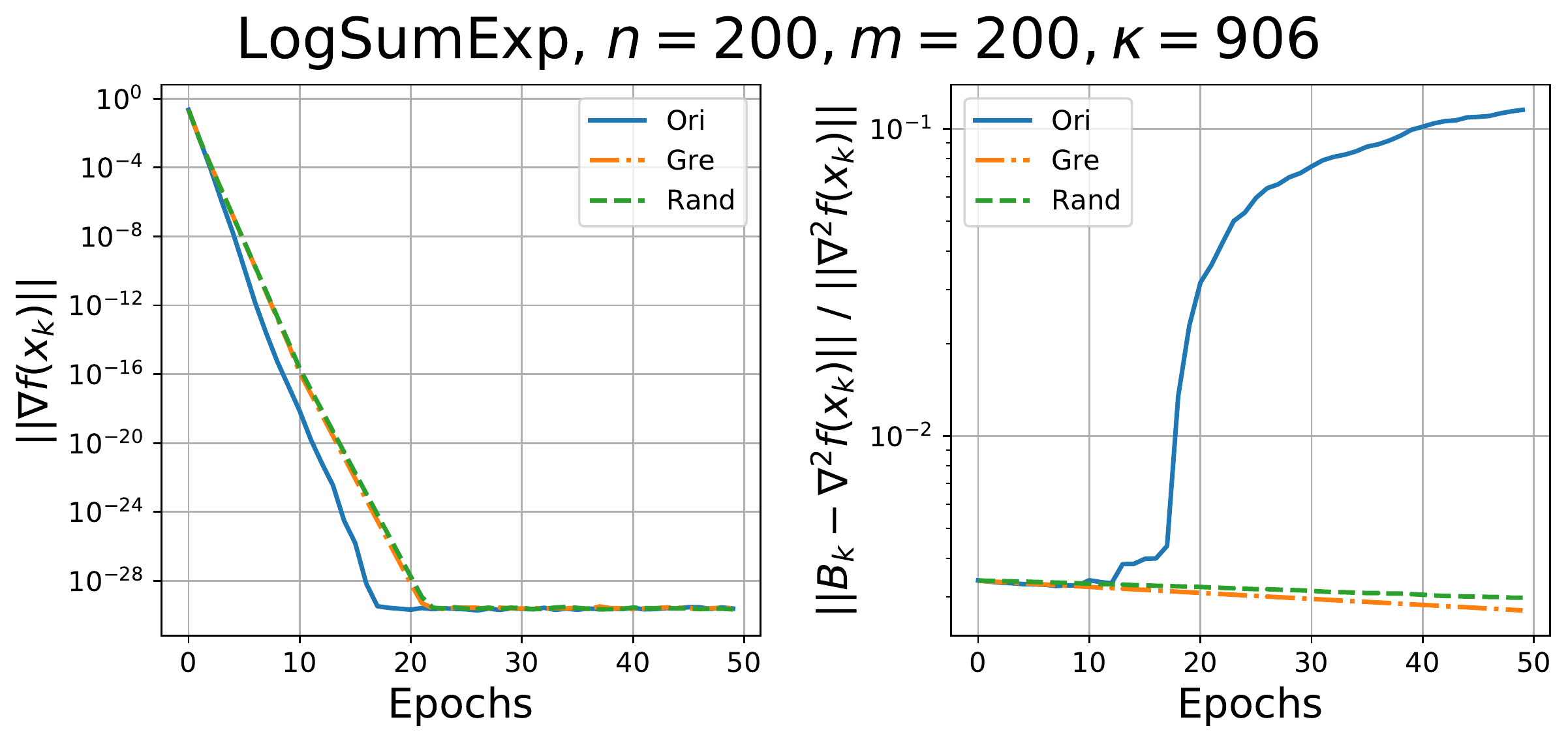}
		\caption{$\mB_0=\nabla^2 f(\vx_0)$.} \label{fig:exp-near}
	\end{subfigure}
	\caption{Comparison of greedy and random with original Broyden's methods for Regularized Log-Sum-Exp under the regularized constant $\gamma=1$, We choose shared $\vx_0\sim\text{Unif}\left(\gS^{n-1}\right)$, and list the dimension $n$ and the condition number $\kappa$ in each title.
		(a) $\mB_0=L\mI_n$, where $L=2\lambda_{\max}(C C^\top)+\gamma$ is the smooth coefficient of $f(\vx)$. (b) $\mB_0=\nabla^2 f(\vx_0)$. In each sub-figure, we plot $\norm{\nabla f(\vx_k)}$ in the left graph and $\norm{\mB_k-\nabla^2 f(\vx_k)}_F / \norm{\nabla^2 f(\vx_k)}_F$ in the right graph. Legends explanation: `Ori' is the original Broyden's method, `Gre' and `Rand' are the greedy and random Broyden's method.}
	\label{fig:broyden-logsumexp-reg}
\end{figure}

Followed by \citet{rodomanov2021greedy}, we present preliminary computational results for the following regularized Log-Sum-Exp function:
\begin{equation*}%\label{eq:log-sum-exp}
f(\vx) := \ln \left(\sum_{j=1}^m e^{\vc_j^\top \vx-b_j}\right)+\frac{1}{2}\sum_{j=1}^m \left(\vc_j^\top\vx\right)^2+\frac{\gamma}{2}\left\|\vx\right\|^2, \vx\in\sR^n,
\end{equation*}
where $\mC = \left(\vc_1,\dots,\vc_m\right)\in\sR^{n\times m}$, $b_1, \dots, b_m \in\sR$, and $\gamma>0$.
%All the standard methods need access only to the gradient of function $f$:
%\begin{equation*}%\label{eq:log-grad}
%\nabla f(\vx) = g(\vx)+\sum_{j=1}^m \left(\vc_j^\top\vx\right) \vc_j +\gamma\vx, g(\vx):=\sum_{j=1}^m \pi_j(\vx)\vc_j,
%\end{equation*}
%where
%\begin{equation*}%\label{eq:pi}
%\pi_j(\vx) := \frac{e^{\vc_j^\top \vx-b_j}}{\sum_{i=1}^me^{\vc_i^\top \vx-b_i}} \in [0, 1], j\in[m].
%\end{equation*}
We apply Broyden's methods to solve $\mF(\vx):=\nabla f(\vx) =\bm{0}$ and we get $\mJ(\vx) = \nabla^2 f(\vx)$.
%Moreover, we have the analytic Hessian expression below:
%\begin{equation*}%\label{eq:log-hessian}
%\nabla^2 f(\vx) = \sum_{j=1}^m \left(\pi_j(\vx)+1\right)\vc_j\vc_j^\top-g(\vx)g(\vx)^\top + \gamma I_n,
%\end{equation*}
%showing that $\gamma \mI_n \preceq \mJ(\vx) \preceq L\mI_n$ where $L:=2\lambda_{\max}(C C^\top)+\gamma$, and the condition number is $\kappa=L/\gamma$.
Moreover, the Hessian $\nabla^2 f(\vx)$ is also differentiable. Thus it is locally Lipschitz continuous restricted in a bounded domain near the optimal solution $\vx^*$, which satisfies Assumption \ref{ass:lisp}.
%Actually, we only need the expression of $\nabla f(\vx)$ in practice, while the above calculation only guarantees such objective satisfies our setting. 
We also adopt the same synthetic data as \cite[Section 5.1]{rodomanov2021greedy}. 
That is, we generate a collection of random vectors $\hat{\vc}_1, \dots, \hat{\vc}_m$ with entries, uniformly distributed in the interval $ [-1, 1] $. Then we generate $ b_1, \dots, b_m$ from the same distribution and we define
\[ \vc_j := \hat{\vc}_j -\nabla \hat{f}(\bm{0}), j=[m] \text{ where } \hat{f}(\vx) := \ln \left(\sum_{j=1}^m e^{\hat{\vc}_j^\top \vx-b_j}\right). \]
%Note that by such construction, we obtain
%\begin{equation*}
%	\nabla f(\bm{0}) = \frac{1}{\sum_{i=1}^m e^{-b_i}} \sum_{j=1}^m e^{-b_j}\left(\hat{\vc}_j-\nabla \hat{f}(\bm{0})\right) =\bm{0}.
%\end{equation*}
So the unique minimizer of our test function is $ \vx^* = \bm{0} $. 
We set the same starting point $\vx_0\sim\mathcal{N}\left(\bm{0}, \mI_n\right)$ and $\mH_0 = \mB_0^{-1}$, where $\mB_0 =s \nabla^2 f(\vx^*)$ or $s\nabla^2 f(\vx_0)$, where scalar $s$ controls the distance with the optimal $\mJ_*=\nabla^2 f(\vx^*)$. We also plot $\norm{\nabla f(\vx_k)}$ to trace the objective and $\frac{\norm{\mB_k-\nabla^2 f(\vx_k)}_F}{\norm{\nabla^2 f(\vx_k)}_F}$ to trace the Hessian approximation error in each epoch in Figure \ref{fig:broyden-logsumexp-reg}. 

As Figure \ref{fig:broyden-logsumexp-reg} depicted, greedy and random Broyden's methods have comparable performance to the original Broyden's method, but they have different suitable schemes.
From Figure \ref{fig:exp-gel}, we observe that greedy and random methods are more suitable for the \textbf{Bad Initial}, while the original Broyden's method is still the best for the \textbf{Good Initial}. 
Moreover, greedy and random methods could give a more stable training curve of $\norm{\nabla f(\vx_k)}$ and a more accurate approximation of the current $\nabla^2 f(\vx_k)$. 
Furthermore,  the greedy Broyden's method converges faster than the random version. 
Although our theory does not apply to such observation, we consider the random-type method may encounter wrong update directions $\vu_k$ during optimization. 
Yet the greedy-type method could generally choose a good $\vu_k$ in each step.
Once the random method selects a good update direction $\vu_k$, then the performance improves significantly, showing a staircase curve in Figure \ref{fig:broyden-logsumexp-reg}.
However, the greedy Broyden's method requires extra computation cost to obtain a good direction $\vu_k$.
Overall, we conclude that greedy and random methods could give a better estimation of current Hessians. And they may have a great scope of suitable initialization.

\paragraph{The Chandrasekhar H-equation.}

\begin{figure}[t]
	\centering
	\hspace{-8pt}
	\begin{subfigure}[b]{0.48\textwidth}
		\includegraphics[width=\linewidth]{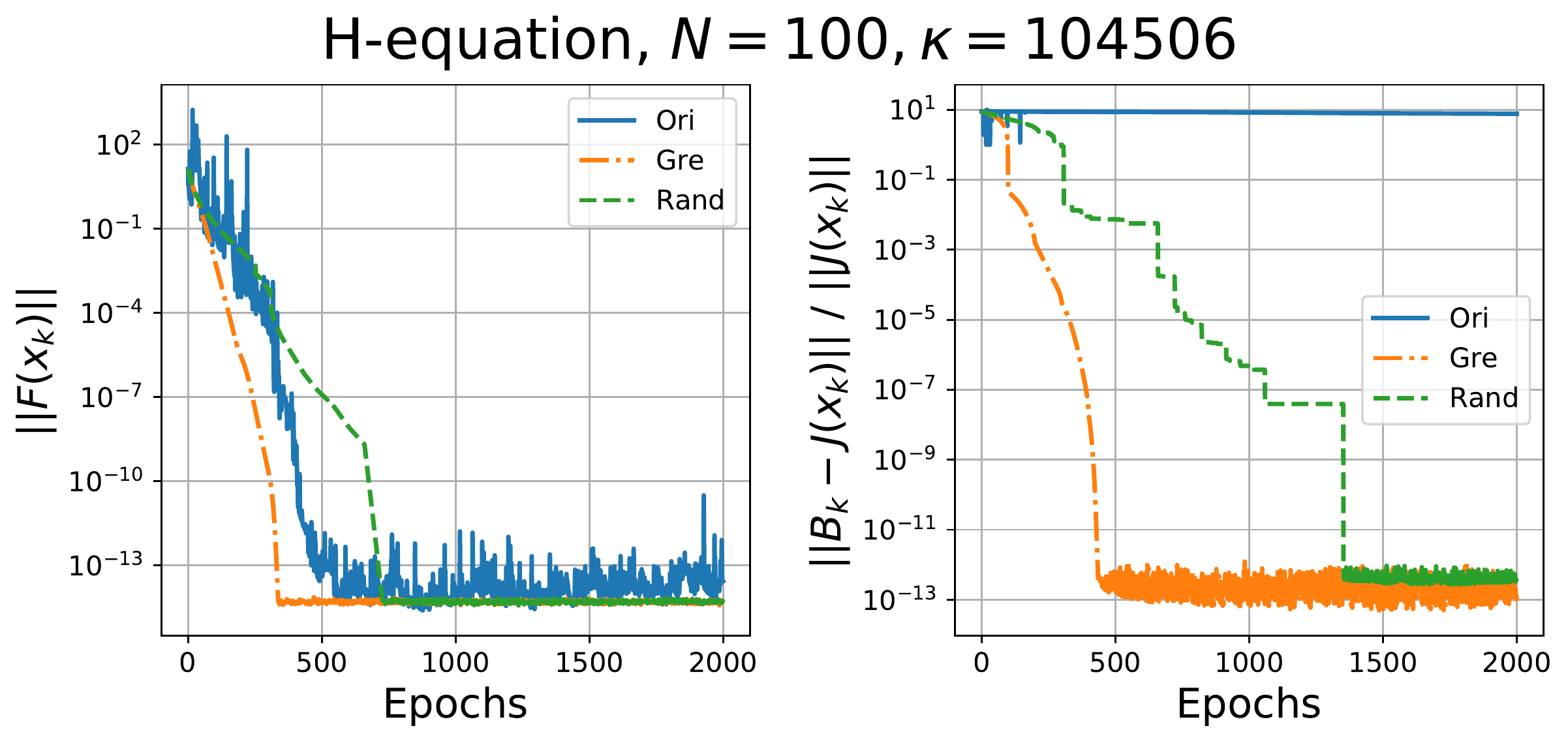}
		\caption{$\mB_0=10\mI_N$.} \label{fig:exp-10I}
	\end{subfigure}
	\hspace{-6pt}
	\begin{subfigure}[b]{0.48\textwidth}
		\includegraphics[width=\linewidth]{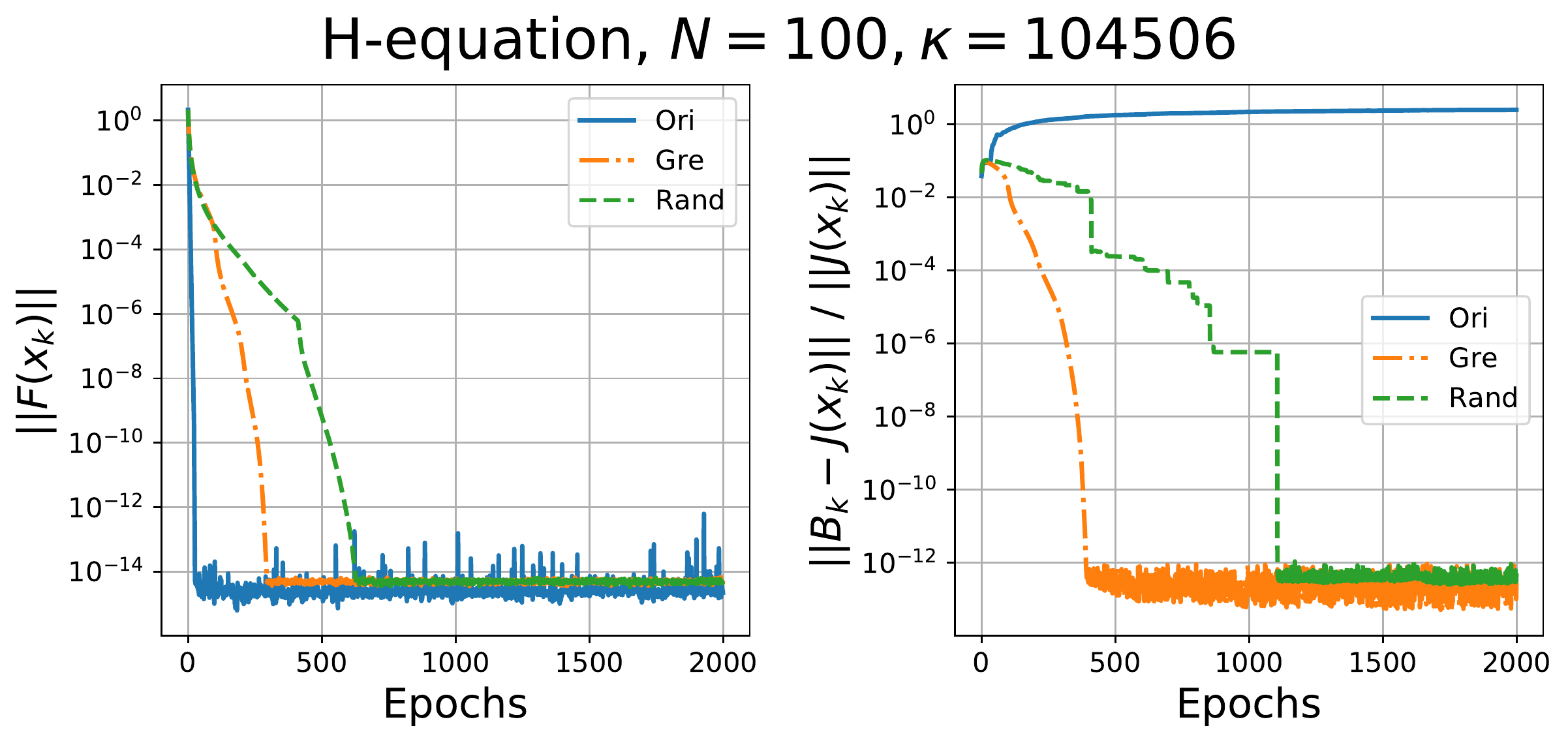}
		\caption{$\mB_0=\mJ(\vx_0)$.} \label{fig:exp-H-near}
	\end{subfigure}
	\caption{Comparison of greedy and random with original Broyden's methods for the Chandrasekhar H-equation. The constant $c=1{-}10^{-10}$ and we choose shared $\vx_0\sim\mathcal{N}\left(\bm{0}, \mI_N\right)$. We also list the dimension $N$ and the condition number $\kappa$ in each title. (a) $\mB_0=10\mI_N$, (b) $\mB_0=\mJ(\vx_0)$. Due to numerical overflow, the original Broyden's method encounter `NAN' during optimization in (a). In each sub-figure, we plot $\norm{\mF(\vx_k)}$ in the left graph and $\norm{\mB_k-\mJ(\vx_k)}_F/\norm{\mJ(\vx_k)}_F$ in the right graph. Legends explanation: `Ori' is the original Broyden's method, `Gre' and `Rand' are the greedy and random Broyden's method.}
	\label{fig:broyden-H-equation}
\end{figure}

We also try to solve a general nonlinear equation by Broyden's methods. The next problem is the Chandrasekhar H-equation \cite[Section 5]{nocedal2006numerical}, which needs to solve the following 
the equation with $\mu_i = (i-0.5)/N, i\in[N]$:
\begin{equation*}
    \mF(\vx)_i = x_i - \left(1-\frac{c}{2N}\sum_{j=1}^N\frac{\mu_ix_j}{\mu_i+\mu_j}\right)^{-1},
\end{equation*}
where $\vx = (x_1, \dots, x_N)^\top \in \sR^N, \mF(\vx) = (\mF(\vx)_1, \dots, \mF(\vx)_N)^\top \in \sR^N$.
%The Jacobian of $\mF(\cdot)$ is $\mJ(\vx) = [\mJ(\vx)_{i k}] \in\sR^{N \times N}$ with
%\begin{equation*}
%    \mJ(\vx)_{i k} = \frac{\partial \mF(\vx)_i}{\partial x_k} = \mathbbm{1}_{\{i=k\}}-\frac{c}{2N} \cdot \frac{\mu_i}{\mu_i+\mu_k} \cdot \left(1-\frac{c}{2N}\sum_{j=1}^N\frac{\mu_ix_j}{\mu_i+\mu_j}\right)^{-2},
%\end{equation*}
%where $\mathbbm{1}_{\{i=k\}}$ is the indicator function that equals to zero only when $i$ equals to $k$.
It is known \cite{mullikin1968some} that the discrete analog has
solutions for $ c \in (0, 1) $. Thus it is differentiable, and locally Lipschitz continuous restricting in a bounded domain near the optimal solution $\vx^*$, which satisfies Assumption \ref{ass:lisp}. We adopt Broyden's method to solve $\mF(\vx) = \bm{0}$, and set the same starting point $\vx_0 = \vx^*+0.1\cdot\norm{\vx^*}\cdot\bm{\epsilon}$, where $\bm{\epsilon} \sim \text{Unif}\left(\gS^{N-1}\right)$ and $\mH_0 = \mB_0^{-1}$, where $\mB_0=s \mJ(\vx_0)$ from practice with various $s\in\sR$. We also plot $\norm{\mF(\vx_k)}$ to trace the objective and $\frac{\norm{\mB_k-\mJ(\vx_k)}_F}{\norm{\mJ(\vx_k)}_F}$ to trace the Hessian approximation error in each epoch in Figure \ref{fig:broyden-H-equation}. 

As Figure \ref{fig:broyden-H-equation} depicted, we discover consistent phenomena to Regularized Log-Sum-Exp problem.
We observe that greedy and random methods are suitable for the ``bad initial'', while the original Broyden's method is suitable for ``good initial''. 
Moreover, greedy and random methods make the training curve more stable and the approximation of $\mJ(\vx_k)$ more accurate.
Hence, we consider the performance of greedy and random Broyden's methods is also good. 

\section{Conclusion} \label{sec:conclusion}
In this work, we have presented two variants of Broyden's method for solving general nonlinear equations, i.e.,  the greedy Broyden's method and random Broyden's method.
We have proved explicit condition-number-free local superlinear convergence rates of these two algorithms. 
Our algorithms can achieve that the approximate Jacobian will converge to the exact Jacobian and the convergence rates are asymptotically faster than the original Broyden's method.
We hope our variants of Broyden's method could give a more quantitative understanding of the family of quasi-Newton methods.
After all, there are still some remaining questions. 
For example, whether one can design novel quasi-Newton methods achieving a faster convergence rate than our greedy Broyden's method.
We leave this as future work.

\pagebreak 
\appendix

\section{Auxiliary Theorems}

\begin{lemma}[Extension of \citet{dennis1996numerical} Theorem 3.1.4]\label{aux-lemma2}
\quad\\
	Let $\|\cdot \|$ be any norm on $\sR^{n\times n}$ that obeys
	$ \|\mA\mB\| \leq \|\mA\|\cdot \|\mB\| $ for any $\mA,\mB \in\sR^{n\times n}$
	and $\|\mI_n\|=1$. 
	Letting $\mE \in\sR^{n\times n}$ with $\norm{\mE} < 1$, then $\left(\mI_n-\mE\right)^{-1}$ exists and 
	\begin{equation*}
		\norm{\left(\mI_n-\mE\right)^{-1}} \leq \frac{1}{1-\norm{\mE}},
	\end{equation*}
	and
	\begin{equation*}
		\norm{\left(\mI_n-\mE\right)^{-1}-\mI_n} \leq \frac{\norm{\mE}}{1-\norm{\mE}}.
	\end{equation*}
\end{lemma}

\begin{proof}
	Since $\norm{\mE} < 1$, we have $0 \leq \liminf_{k\to \infty}\norm{\mE^k} \leq \limsup_{k\to \infty}\norm{\mE^k} \leq \lim_{k\to \infty}\norm{\mE}^k =0 $, thus $\lim_{k\to \infty}\norm{\mE^k}=0$ and $\lim_{k\to \infty} \mE^k = \bm{0}$.
	Moreover, we could get $\lim_{k\to \infty} \left(\mI_n-\mE\right)\left(\sum_{i=0}^k \mE^i\right) = \lim_{k\to \infty} \mI_n-\mE^{k+1} = \mI_n$. Hence
	\begin{equation*}%\label{eq:inv-i-e}
		\left(\mI_n-\mE\right)^{-1} = \sum_{i=0}^\infty \mE^i = \mI_n+\mE+\mE^2+\dots.
	\end{equation*}
	Therefore $\left(\mI_n-\mE\right)^{-1}$ exists, and
	\begin{equation*}
		\norm{\left(\mI_n-\mE\right)^{-1}} = \norm{\sum_{i=0}^\infty \mE^i} \leq \sum_{i=0}^\infty \norm{ \mE^i} \leq \sum_{i=0}^\infty \norm{ \mE}^i = \frac{1}{1-\norm{\mE}},
	\end{equation*}
	and
	\begin{equation*}
		\norm{\left(\mI_n-\mE\right)^{-1}-\mI_n} = \norm{\sum_{i=1}^\infty \mE^i} \leq \sum_{i=1}^\infty \norm{\mE^i}  \leq \sum_{i=1}^\infty \norm{ \mE}^i = \frac{\norm{\mE}}{1-\norm{\mE}}.
	\end{equation*}
	The proof is finished.
\end{proof}

\section{Proof of Proposition \ref{prop:Broyden}}
\begin{proof}
	The Broyden's update rule gives
	\begin{equation*}
	\mB_+-\mA=\mB+\frac{\left(\mA-\mB\right)\vu\vu^\top}{\vu^\top\vu}-\mA = \left(\mB-\mA\right)\left(\mI_n-\frac{\vu\vu^\top}{\vu^\top\vu}\right).
	\end{equation*}
	Hence we get
	\[ \overline{\mB}_+ = \overline{\mB} \left(\mI_n-\frac{\vu\vu^\top}{\vu^\top\vu}\right). \]
	Therefore,
	\begin{equation*}%\label{eq:br-up}
	\begin{aligned}
	\overline{\mB}_+\overline{\mB}_+^\top &= \overline{\mB} \left(\mI_n-\frac{\vu\vu^\top}{\vu^\top\vu}\right) \left(\mI_n-\frac{\vu\vu^\top}{\vu^\top\vu}\right)^\top\overline{\mB}^\top = \overline{\mB} \left(\mI_n - \frac{\vu\vu^\top}{\vu^\top\vu}\right) \overline{\mB}^\top \\ 
	&= \overline{\mB}\overline{\mB}^\top-  \frac{\overline{\mB}\vu\vu^\top\overline{\mB}^\top}{\vu^\top\vu}.
	\end{aligned}
	\end{equation*}
 Then Eq.~\eqref{bro-up-prop} is proved.
	
	Now  we take the trace in both sides of Eq.~\eqref{bro-up-prop}, we obtain
	\[ \norm{\overline{\mB}_+}_F^2 = \norm{\overline{\mB}}_F^2 -   \frac{\norm{\overline{\mB}\vu}^2}{\norm{\vu}^2} \leq \norm{\overline{\mB}}_F^2. \]
	Moreover, from Eq.~\eqref{bro-up-prop}, for any $\vv\in\sR^n$, with $\norm{\vv}=1$, we get
	\[ \norm{\overline{\mB}_+^\top\vv}^2 = \norm{\overline{\mB}^\top\vv}^2 - \frac{\left(\vv^\top\overline{\mB}\vu\right)^2}{\norm{\vu}^2} \leq \norm{\overline{\mB}^\top\vv}^2 \leq \norm{\overline{\mB}}^2. \]
	Then we obtain
	\[ \norm{\overline{\mB}_+}^2 = \sup_{\norm{\vv}=1, \vv\in\sR^n} \norm{\overline{\mB}_+^\top\vv}^2 \leq \norm{\overline{\mB}}^2. \]
	Thus, Eq.~\eqref{bro-up-prop-ex} is proved.
\end{proof}

Using Assumptions \ref{ass:lisp}, we could give the estimator of objective and its Jacobian around $\vx^*$.
\begin{lemma}\label{lemma:prop}
	Under Assumption \ref{ass:lisp}, and letting $\mJ = \int_{0}^{1} \mJ\left(\vx+t\vu\right)dt$ for some $\vx, \vu \in\sR^n$, then the following inequalities hold:
	\begin{align}
		\|\mF(\vx) - \mJ(\vx^*)(\vx-\vx^*) \| \leq \frac{M}{2} \norm{ \vx-\vx^* }^2. \label{eq:lemma-j2}
	\end{align}		
\end{lemma}

\begin{proof}
	We denote $\vs = \vx-\vx^*$. Then from Assumption \ref{ass:lisp}, we get
	\begin{equation*}
		\begin{aligned}
			\norm{\mF(\vx) - \mJ(\vx^*)(\vx-\vx^*)} &= \norm{\mF(\vx) - \mF(\vx^*) - \mJ(\vx^*)(\vx-\vx^*)} = \norm{ \int_{0}^{1} \left[\mJ\left(\vx^*+t\vs\right) - \mJ(\vx^*)\right]\vs dt } \\
			&\leq \int_{0}^{1} \norm{\mJ\left(\vx^*+t\vs\right) - \mJ(\vx^*)} \cdot \norm{\vs} dt \stackrel{(\ref{eq:lisp})}{\leq}  \int_{0}^{1} Mt \norm{\vs}^2 dt = \frac{M}{2} \norm{ \vx-\vx^*}^2,
		\end{aligned}
	\end{equation*}
	which shows that Eq.~\eqref{eq:lemma-j2} holds. 
\end{proof}

\section{Missing Proofs}
\subsection{Proof of Lemma \ref{lemma:b-progress}}
\begin{proof}
	For the \textit{greedy update}, from Proposition \ref{prop:Broyden} with $\mC=\mI_n$, we obtain
	\begin{equation*}
	\left(\mB_+-\mA\right)\left(\mB_+-\mA\right)^\top = \left(\mB-\mA\right)^\top\left(\mB-\mA\right)-  \frac{\left(\mB-\mA\right)\vu\vu^\top\left(\mB-\mA\right)^\top}{\vu^\top\vu}.
	\end{equation*}
	Taking the trace in both sides and using Eq.~\eqref{eq:greedy-u}, we get
	\begin{equation*}
	\begin{aligned}
	\norm{\mB_+-\mA}_F^2 &= \norm{\mB-\mA}_F^2 - 
	\max_{\vu \in \{\ve_1, \dots, \ve_n\}} \frac{\vu^\top\left(\mB-\mA\right)^\top \left(\mB-\mA\right)\vu}{\vu^\top\vu} \\
	&\leq \norm{\mB-\mA}_F^2 -  \frac{1}{n}\sum_{i=1}^n \frac{\ve_i^\top\left(\mB-\mA\right)^\top \left(\mB-\mA\right)\ve_i}{\ve_i^\top\ve_i} \\
	&=\norm{\mB-\mA}_F^2 - \tr\left[\left(\mB-\mA\right)^\top \left(\mB-\mA\right) \cdot \frac{1}{n}\sum_{i=1}^n \ve_i\ve_i^\top\right] \\
	&= \left(1-\frac{1}{n}\right) \norm{\mB-\mA}_F^2
	\end{aligned}
	\end{equation*}
	Therefore, the greedy choice of $\vu$ leads to
	\[ \norm{\mB_+-\mA}_F^2 \leq \left(1-\frac{1}{n}\right) \norm{\mB-\mA}_F^2. \]
	For the \textit{random update}, noting that $\tr\left(\E_{\vu}\frac{\vu}{\norm{\vu}} \cdot \frac{\vu^\top}{\norm{\vu}}\right) = 1$ , we obtain
	\begin{equation*}
	\E_{\vu} \left[ \frac{\vu \vu^\top}{\vu^\top\vu} \right] = \E_{\vu} \left[ \frac{\vu}{\norm{\vu}} \cdot \frac{\vu^\top}{\norm{\vu}}\right] = \frac{1}{n} \mI_n.
	\end{equation*}
	Hence, we obtain 
	\begin{equation*}
	\begin{aligned}
	\E_{\vu} \norm{\mB_+-\mA}_F^2 &= \norm{\mB-\mA}_F^2 -
	\E_{\vu} \frac{\vu^\top\left(\mB-\mA\right)^\top \left(\mB-\mA\right)\vu}{\vu^\top\vu} \\
	&= \norm{\mB-\mA}_F^2 -  \tr\left[\left(\mB-\mA\right)^\top \left(\mB-\mA\right) \E_{\vu} \frac{\vu\vu^\top}{\vu^\top\vu}\right] \\
	&=\norm{\mB-\mA}_F^2 - \tr\left[\left(\mB-\mA\right)^\top \left(\mB-\mA\right) \cdot \frac{1}{n}\mI_n\right] \\
	&= \left(1-\frac{1}{n}\right) \norm{\mB-\mA}_F^2
	\end{aligned}
	\end{equation*}
	Therefore, the random choice of $\vu$ leads to
	\[ \E_{\vu} \norm{\mB_+-\mA}_F^2 = \left(1-\frac{1}{n}\right) \norm{\mB-\mA}_F^2. \]
\end{proof}

\subsection{Proof of Lemma \ref{lemma:sigma-up}}
\begin{proof}
	Following the definition of $\sigma_{k+1}$, we have 
	\begin{equation}\label{eq:bde}
	\mB_{k+1} -\mJ_{k+1} \stackrel{(\ref{eq:x-up})}{=} \mB_k +  \frac{\left(\mJ_{k+1}-\mB_k\right)\vu_k\vu_k^\top}{\vu_k^\top\vu_k}-\mJ_{k+1} = \left(\mB_k-\mJ_{k+1}\right)\left[ \mI_n- \frac{\vu_k\vu_k^\top}{\vu_k^\top\vu_k}\right].
	\end{equation}
	Take $\norm{\cdot}_F$ in both sides of Eq. (\ref{eq:bde}) and take expectation only on $\vu_{k}$ with fixed $\vu_1, \dots, \vu_{k-1}$, we get
	\begin{equation*}%\label{eq:part1}
	\begin{aligned}
	\E_{\vu_{k}} \left[\sigma_{k+1}^2\right] & \stackrel{(\ref{jac-measure})}{=} \E_{\vu_{k}}\left[ \norm{\mB_{k+1}-\mJ_{k+1}}_F^2\right] \stackrel{(\ref{eq:bde})}{=} 
	\E_{\vu_{k}} \left[\norm{\left(\mB_k-\mJ_{k+1}\right)\left[ \mI_n- \frac{\vu_k\vu_k^\top}{\vu_k^\top\vu_k}\right]}_F^2\right].
	\end{aligned}
	\end{equation*}
	From Proposition \ref{prop:Broyden}, we obtain
	\begin{equation}\label{eq:s-up-part1}
	\norm{\left(\mB_k-\mJ_{k+1}\right)\left[ \mI_n- \frac{\vu_k\vu_k^\top}{\vu_k^\top\vu_k}\right]}_F^2 \leq \norm{\mB_k-\mJ_{k+1}}_F^2.
	\end{equation}
	From Lemma \ref{lemma:b-progress}, we obtain
	\begin{equation}\label{eq:s-up-part11}
	\E_{\vu_{k}} \left[ \norm{\left(\mB_k-\mJ_{k+1}\right)\left[ \mI_n- \frac{\vu_k\vu_k^\top}{\vu_k^\top\vu_k}\right]}_F^2 \right]\leq \left(1-\frac{1}{n}\right) \norm{\mB_k-\mJ_{k+1}}_F^2.
	\end{equation}
	Note that
	\begin{equation}\label{eq:s-up-part2}
	\begin{aligned}
	\norm{\mB_k-\mJ_{k+1}}_F^2 &= \norm{\mB_k-\mJ_{k}+\mJ_k-\mJ_{k+1}}_F^2 \\
	&\leq \norm{\mB_k-\mJ_{k}}_F^2 + 2\norm{\mB_k-\mJ_{k}}_F\cdot \norm{\mJ_k-\mJ_{k+1}}_F+\norm{\mJ_k-\mJ_{k+1}}_F^2 \\
	&=\sigma_k^2 + 2\sigma_k \cdot \norm{\mJ_k-\mJ_{k+1}}_F+\norm{\mJ_k-\mJ_{k+1}}_F^2.
	\end{aligned}
	\end{equation}
	By Lemma \ref{lemma:prop}, we get
	\begin{equation}\label{eq:jk}
	\norm{\mJ_{k}-\mJ_{k+1}}_F \leq \sqrt{n} \norm{\mJ_{k}-\mJ_{k+1}} \leq \sqrt{n}\left[\norm{\mJ_{k}-\mJ_*}+\norm{\mJ_{k+1}-\mJ_*}\right] \stackrel{(\ref{eq:lisp})}{\leq} M\sqrt{n}\left(r_k+r_{k+1}\right).
	\end{equation} 
	Combining Eq.~\eqref{eq:s-up-part1}, \eqref{eq:s-up-part2}, and \eqref{eq:jk}), we get
	\begin{equation*}
	\sigma_{k+1}^2 \leq \sigma_k^2 + 2\sigma_k \sqrt{n}M\left(r_k+r_{k+1}\right)+n M^2\left(r_k+r_{k+1}\right)^2.
	\end{equation*}
	Combining Eq.~\eqref{eq:s-up-part11}, \eqref{eq:s-up-part2}, and \eqref{eq:jk}, we obtain
	\begin{equation*}
	\E_{\vu_{k}} \left[\sigma_{k+1}^2\right] \leq \left(1-\frac{1}{n}\right)\sigma_k^2 + 2\sigma_k \sqrt{n}M\left(r_k+r_{k+1}\right)+n M^2\left(r_k+r_{k+1}\right)^2.
	\end{equation*}
\end{proof}

\subsection{Proof of Lemma \ref{lemma:r-up}}
\begin{proof}
	Note that
	\begin{equation*}
	\begin{aligned}
	\vs_{k+1} &\stackrel{(\ref{var-defn})}{=} \vx_{k+1}-\vx^* \stackrel{(\ref{eq:x-up})}{=}\vx_{k}-\mB_k^{-1}\mF(\vx_k)-\vx^* = \mB_k^{-1}\bigg[\mB_k \left(\vx_k-\vx^*\right)-\mF(\vx_k)\bigg]\\
	&= \mB_k^{-1}\bigg[\left(\mB_k-\mJ_{*}\right) \left(\vx_k-\vx^*\right)-\left[\mF(\vx_k)-\mJ_{*}\left(\vx_k-\vx^*\right)\right]\bigg].
	\end{aligned}
	\end{equation*}	
	Now we take the norm $\norm{\cdot}$ in both sides of above equation which leads to 
	\begin{equation}\label{eq:rk-up1}
	\begin{aligned}
	r_{k+1}
	&\leq \norm{\mB_k^{-1}\mJ_*} \cdot \norm{\mJ_*^{-1}} \cdot \bigg[ \left(\norm{\mB_k-\mJ_k}+\norm{\mJ_k-\mJ_*}\right) \cdot \norm{\vx_k-\vx^*} + \norm{\mF(\vx_k)-\mJ_{*} \left(\vx_k-\vx^*\right)} \bigg] \\
	&\stackrel{(\ref{jac-measure}, \ref{eq:lisp})}{\leq} c\norm{\mB_k^{-1}\mJ_*} \cdot \bigg[\left(\sigma_k+Mr_k\right) r_k + \norm{\mF(\vx_k)-\mJ_{*} \left(\vx_k-\vx^*\right)} \bigg] \\
	&\stackrel{(\ref{eq:lemma-j2})}{\leq} c\norm{\mB_k^{-1}\mJ_*} \cdot \left[\left(\sigma_k+Mr_k\right) r_k +\frac{M r_k^2}{2}\right].
	\end{aligned}
	\end{equation}
	If $c\left(\sigma_k+Mr_k\right)<1$, we get
	\[ \norm{\mI_n-\mJ_{*}^{-1}\mB_k} = \norm{\mJ_{*}^{-1}\left(\mB_k-\mJ_{*}\right)} \leq \norm{\mJ_{*}^{-1}}\cdot \left[\norm{\mB_k-\mJ_k}+\norm{\mJ_k-\mJ_{*}} \right] \stackrel{(\ref{jac-measure},\ref{eq:lisp})}{\leq} c\left(\sigma_k+Mr_k\right) < 1. \]
	Hence from Lemma \ref{aux-lemma2}, we obtain $\mJ_{*}^{-1}\mB_k$ is nonsingular and 
	\begin{equation}\label{eq:inv}
	\norm{\mB_k^{-1}\mJ_*} \leq \frac{1}{1-\norm{\mI_n-\mJ_{*}^{-1}\mB_k}} \leq \frac{1}{1-c\left(\sigma_k+Mr_k\right)}, \text{ if } c\left(\sigma_k+Mr_k\right)<1.
	\end{equation}
	Finally, we get
	\begin{equation*}%\label{eq:rk-up2}
	r_{k+1}\stackrel{(\ref{eq:rk-up1}, \ref{eq:inv})}{\leq} \frac{3cMr_k/2+c\sigma_k}{1-c\left(\sigma_k+Mr_k\right)}r_k, \text{ if } c\left(\sigma_k+Mr_k\right)<1.
	\end{equation*}
\end{proof}

\subsection{Proof of Lemma \ref{lemma:base}}
\begin{proof}
	From Lemma \ref{lemma:sigma-up}, we have derived the update of $\sigma_{k}$.
	\begin{equation}\label{eq:sk-up}
	\sigma_{k+1}^2 \stackrel{(\ref{eq:sigma-update1})}{\leq} \sigma_k^2 + 2\sigma_k \sqrt{n}M\left(r_k+r_{k+1}\right)+n M^2\left(r_k+r_{k+1}\right)^2.
	\end{equation}
	From Lemma \ref{lemma:r-up}, we have obtained the update of $r_{k}$.
	\begin{equation}\label{eq:rk-up}
	r_{k+1} \stackrel{(\ref{eq:r-update})}{\leq} \frac{3cMr_k/2+c\sigma_k}{1-c\left(\sigma_k+Mr_k\right)}r_k, \text{ if } c\left(\sigma_k+Mr_k\right)<1.
	\end{equation}
	Now we prove by induction. 
	
	1) When $k=0$, we have $c\sigma_0 + 3cMr_0 / 2 \stackrel{(\ref{eq:init-cond})}{\leq} c\sigma_0 + \left(\frac{q}{q+1}-c\sigma_0 \right) = \frac{q}{q+1}<1$, showing that Eq.~\eqref{eq:rk-up} holds and 
	\[ r_{1} \leq \frac{3cMr_0/2+c\sigma_0}{1-c\sigma_0-cMr_0}r_0 \leq  \frac{3cMr_0/2+c\sigma_0}{1-c\sigma_0-3cMr_0/2}r_0\stackrel{(\ref{eq:init-cond})}{\leq} \frac{\frac{q}{1+q}}{1-\frac{q}{1+q}} r_0 = q r_0. \]
	Thus the results hold.
	
	2) Let $k \geq 0$, and suppose that Eq.~\eqref{eq:sigma-bound}, and \eqref{eq:linear-con} hold for all indices up to $ k$.
	Then it holds that $r_k\leq q^kr_0 \leq r_0 $, and 
	\begin{equation}\label{eq:sigma-upper}
	c^2\sigma_k^2 \stackrel{(\ref{eq:sigma-bound})}{\leq} c^2\sigma_0^2 +\frac{1+q}{1-q}\left(2c\sqrt{n}Mr_0+ nc^2M^2r_0^2\right) = c^2\sigma_0^2 + \frac{1+q}{1-q}\left(2a+a^2\right),
	\end{equation}
	where we take $a := \sqrt{n}cMr_0$, and from inductive assumption,
	\begin{equation}\label{eq:a}
	a \stackrel{(\ref{eq:init-cond})}{<} \frac{q}{2(q+1)}.
	\end{equation}	
	When $t\leq \frac{q(1-q)}{12}\leq \frac{q(1-q)}{3(q+1)^2}<1$, we have
	\begin{equation}\label{p1}
	\frac{1+q}{1-q}\left(2t+\frac{q}{q+1}t^2\right) + \frac{4tq}{\sqrt{n}(q+1)} \leq \frac{2(1+q)+q}{1-q} \cdot t+ \frac{t}{1-q} = \frac{3(q+1)t}{1-q} \leq \frac{q}{1+q}.
	\end{equation}
	Setting $b := t\left(\frac{q}{q+1}-c\sigma_0\right) \leq \frac{q}{q+1}t$, and multiplying $\left(\frac{q}{q+1}-c\sigma_0\right)$ in both sides of Eq.~\eqref{p1}, we obtain 
	\[ \frac{1+q}{1-q}\left(2b+b^2\right) + \frac{4b q}{\sqrt{n}(q+1)} \leq \frac{q}{1+q}\left(\frac{q}{1+q}-c\sigma_0\right)\stackrel{(\ref{eq:init-cond})}{\leq} \left(\frac{q}{1+q}\right)^2-c^2\sigma_0^2, \]
	leading to
	\[ \left(\frac{q}{q+1}-\frac{2b}{\sqrt{n}}\right)^2 \geq c^2\sigma_0^2 + \frac{1+q}{1-q}\left(2b+b^2\right). \]
	Hence we obtain when $a \leq \frac{q(1-q)}{12} \left(\frac{q}{1+q} - c\sigma_0\right)$, 
	\begin{equation}\label{eq:only}
	\begin{aligned}
	\left(\frac{q}{q+1}-\frac{2a}{\sqrt{n}}\right)^2 \geq c^2\sigma_0^2 + \frac{1+q}{1-q}\left(2a+a^2\right)
	\stackrel{(\ref{eq:a})}{\Leftrightarrow}
	\frac{2a}{\sqrt{n}}+\sqrt{c^2\sigma_0^2 + \frac{1+q}{1-q}\left(2a+a^2\right)} \leq \frac{q}{q+1}.
	\end{aligned}
	\end{equation}
	This shows that
	\begin{equation}\label{eq:cmr-si}
	cMr_k + c\sigma_k \leq 2cMr_k + c\sigma_k \leq 2cMr_0 + c\sigma_k \stackrel{(\ref{eq:sigma-upper})}{\leq } \frac{2a}{\sqrt{n}}+\sqrt{c^2\sigma_0^2 + \frac{1+q}{1-q}\left(2a+a^2\right)} \stackrel{(\ref{eq:only})}{\leq} \frac{q}{q+1}<1.
	\end{equation}
	Thus Eq.~\eqref{eq:rk-up} holds and 
	\[ r_{k+1} \stackrel{(\ref{eq:rk-up})}{\leq} \frac{3cMr_k/2+c\sigma_k}{1-c\left(\sigma_k+Mr_k\right)} r_k \leq \frac{(q+2)cMr_k+c\sigma_k}{1-\left(cMr_k + c\sigma_k\right)} r_k \stackrel{(\ref{eq:cmr-si})}{\leq} qr_k. \]
	Furthermore, from Eq.~\eqref{eq:sk-up} and $c\sigma_i < 1, \forall i \leq k$, we could obtain 
	\begin{equation*}
	\begin{aligned}
	c^2\sigma_{k+1}^2 &\stackrel{\eqref{eq:sk-up}}{\leq} c^2\sigma_k^2 + 2c^2\sigma_k \sqrt{n}M\left(r_k+r_{k+1}\right)+n c^2M^2\left(r_k+r_{k+1}\right)^2 \\
	&\leq c^2\sigma_k^2 + 2c \sqrt{n}M(q+1)r_k+n c^2M^2(q+1)^2r_k^2 \leq ... \\ &\leq c^2\sigma_0^2 + 2c\sqrt{n}M(q+1)\sum_{i=0}^{\infty} r_i + nc^2M^2(q+1)^2\sum_{i=0}^{\infty} r_i^2 \\
	&= c^2\sigma_0^2 + \frac{1+q}{1-q}\left(2c\sqrt{n}Mr_0+ nc^2M^2r_0^2\right) = c^2\sigma_0^2 + \frac{1+q}{1-q}\left(2a+a^2\right) \\
	&\stackrel{\eqref{eq:only}}{\leq} \left(\frac{q}{q+1}-\frac{2a}{\sqrt{n}}\right)^2 \stackrel{\eqref{eq:a}}{\leq} \left(\frac{q}{q+1}\right)^2.
	\end{aligned}
	\end{equation*}
	Hence $c^2\sigma_{k+1}^2 \leq c\sigma_0^2 + \frac{1+q}{1-q}\left(2c\sqrt{n}Mr_0+ nc^2M^2r_0^2\right) \leq \frac{q^2}{(1+q)^2}$,
	showing that Eq.~\eqref{eq:sigma-bound} holds for $k+1$.
\end{proof}

\subsection{Proof of Corollary \ref{coro:good}}
\begin{proof}
	To give an upper and lower bound of $q_m$, we consider 
	\begin{equation*}%\label{eq:fq}
	f(q) := q(1-q)\left(\frac{q}{1+q}-c\sigma_0\right), \ \frac{q}{q+1} \geq c\sigma_0.
	\end{equation*}
	Note that $q(1-q)$ is increasing when $q \leq \frac{1}{2}$, and $\frac{q}{q+1}$ is also 
	increasing when $q>0$. Hence $f(q)$ is increasing when $\frac{1}{2} \geq q \geq \frac{c\sigma_0}{1-c\sigma_0} \geq c\sigma_0$. Moreover, we have
	\begin{equation*}
	f\left(\frac{1}{2}\right) = \frac{1}{4} \left(\frac{1}{3}-c\sigma_0\right) \stackrel{(\ref{eq:init-super})}{\geq} 12\sqrt{n} c M r_0,
	\end{equation*}
	showing that $q_m \leq \frac{1}{2}$. Then we obtain
	\begin{equation*}
	12 \sqrt{n} c M r_0 = f(q_m) \geq \frac{q_m}{2}\left(\frac{q_m}{2}-c\sigma_0\right), q_m \geq c\sigma_0,
	\end{equation*}
	which gives
	\begin{equation}\label{eq:qm-upper}
	q_m \leq c\sigma_0 +\sqrt{c^2\sigma_0^2+48\sqrt{n}c M r_0} \leq 2c\sigma_0+7\sqrt{\sqrt{n}c M r_0} \leq 7\left(c\sigma_0+\sqrt{\sqrt{n}c M r_0}\right).
	\end{equation}
	Furthermore, 
	\begin{equation*}
	12\sqrt{n}c M r_0 = f(q_m) \leq q_m\left(q_m-c\sigma_0\right), q_m \geq c\sigma_0,
	\end{equation*}
	which gives
	\begin{equation}\label{eq:qm-lower}
	q_m \geq \frac{c\sigma_0+ \sqrt{c^2\sigma_0^2+48\sqrt{n}c M r_0}}{2} \geq \frac{c\sigma_0}{2}+3\sqrt{\sqrt{n}c M r_0} \geq 0.5\left(c\sigma_0+\sqrt{\sqrt{n}c M r_0}\right).
	\end{equation}
	Combining Eq.~\eqref{eq:qm-upper}  and \eqref{eq:qm-lower}, we obtain $q_m = \Theta\left(c\sigma_0+\sqrt{\sqrt{n}c M r_0}\right)$.
\end{proof}

\begin{remark}
	We note that Lemma \ref{lemma:base} has shown that 
	\[ \sqrt{n} c M r_0  \leq \frac{q(1-q)}{12} \left(\frac{q}{q+1}-c\sigma_0\right) \leq \frac{1}{48} \left(\frac{q}{q+1}-c\sigma_0\right), \]
	showing that  $48\sqrt{n}cMr_0 + c\sigma_0 \leq \frac{q}{q+1} = \Theta(1)$. Hence we could assume Eq. \eqref{eq:init-super} in Corollary \ref{coro:good} holds without loss of generality.
\end{remark}

%\nocite{*}
\pagebreak
\bibliography{reference}
\bibliographystyle{plainnat}

\end{document}